\documentclass[12pt]{amsart}
\usepackage{SSdefn}
\usepackage{eucal}
\usepackage{comment}
\usepackage{enumitem}
\usepackage{tikz}

\DeclareMathOperator{\Split}{Split}

\DeclareMathOperator{\Fact}{Fact}

\DeclareMathOperator{\Ext}{Ext}

\DeclareMathOperator{\sdim}{sdim}

\newcommand{\Fl}{\mathbf{Fl}}

\newcommand{\gl}{\mathfrak{gl}}
\newcommand{\ord}{\mathrm{ord}}

\newcommand{\sing}{\mathrm{sing}}

\newcommand{\Fil}{\mathrm{Fil}}
\newcommand{\univ}{\mathrm{univ}}
\newcommand{\cq}{/\!\!/}
\newcommand{\Syl}{\mathrm{Syl}}

\newcommand{\stacks}[1]{\cite[Tag \href{https://stacks.math.columbia.edu/tag/#1}{#1}]{stacks-project}}

\title[Flag supervarieties and determinantal ideals]{Cohomology of flag supervarieties and\\ resolutions of determinantal ideals}
\date{August 28, 2021}

\author{Steven V Sam}
\address{Department of Mathematics, University of California, San Diego, CA}
\email{\href{mailto:ssam@ucsd.edu}{ssam@ucsd.edu}}
\urladdr{\url{http://math.ucsd.edu/~ssam/}}
\thanks{SS was supported by NSF grant DMS-1812462.}

\author{Andrew Snowden}
\address{Department of Mathematics, University of Michigan, Ann Arbor, MI}
\email{\href{mailto:asnowden@umich.edu}{asnowden@umich.edu}}
\urladdr{\url{http://www-personal.umich.edu/~asnowden/}}
\thanks{AS was supported by NSF grant DMS-1453893.}

\begin{document}

\begin{abstract}
We study the coherent cohomology of generalized flag supervarieties. Our main observation is that these groups are closely related to the free resolutions of (certain generalizations of) determinantal ideals. In the case of super Grassmannians, we completely compute the cohomology of the structure sheaf: it is composed of the singular cohomology of a Grassmannian and the syzygies of a determinantal variety. The majority of the work involves studying the geometry of an analog of the Grothendieck--Springer resolution associated to the super Grassmannian; this takes place in the world of ordinary (non-super) algebraic geometry. Our work gives a conceptual explanation of the result of Pragacz--Weyman that the syzygies of determinantal ideals admit an action of the general linear supergroup. In a subsequent paper, we will treat other flag supervarieties in detail.
\end{abstract}

\maketitle
\tableofcontents

\section{Introduction}

Super geometry is a compelling generalization of algebraic geometry, with important connections to physics and pure mathematics. However, a number of fundamental objects in super geometry are still poorly understood. For example, the cohomology of natural vector bundles on flag supervarieties is not known in general. In this paper, we develop a general method for attacking this problem in some new cases, and use it to completely compute the cohomology of the structure sheaf on the super Grassmannian.

\subsection{General approach} \label{ss:intro-gen}

Let $G$ be a complex reductive supergroup, let $P$ be a parabolic subsupergroup, and let $X=G/P$ be the associated flag supervariety. We outline a general approach to study $\rH^*(X, \cO_X)$.

Let $G_0$, $P_0$, and $X_0$ be the reduced subschemes of $G$, $P$, and $X$. Then $G_0$ is a complex reductive group, $P_0$ is a parabolic subgroup, and $X_0=G_0/P_0$ is a flag variety. Let $\cI$ be the ideal sheaf defining $X_0$ inside of $X$. Since $X$ is a smooth supervariety, $\cI/\cI^2$ is a locally free coherent sheaf on $X_0$, and $\gr(\cO_X)$ (formed with respect to the $\cI$-adic filtration) is the exterior algebra on $\cI/\cI^2$. It follows that we have a spectral sequence
\begin{displaymath}
\rE_1^{p,q} = \rH^{p+q}(X_0, \lw^p(\cI/\cI^2)) \implies \rH^{p+q}(X, \cO_X).
\end{displaymath}
This is our primary tool for connecting the cohomology of supervarieties to ordinary varieties.

To make use of this spectral sequence, we need to understand the cohomology of $\lw(\cI/\cI^2)$. The genesis of this paper was the observation that the cohomology of exterior algebras, especially on flag varieties, appears in another context: namely, the calculation of syzygies of determinantal varieties (and similar varieties) via the geometric method developed by Kempf, Lascoux, and Weyman, among others (see \cite{weyman} for an exposition). This allows us to relate the cohomology of $\lw(\cI/\cI^2)$ to syzygies, in certain cases.

We now explain how this works. Let $\fg$ and $\fp$ be the Lie superalgebras of $G$ and $P$. For a point $x=gP_0$ of $X_0$, let $\fp^x$ be the Lie superalgebra of $gPg^{-1}$. Let $Y$ be the vector bundle over $X_0$ whose fiber over $x$ is $\fp^x_1$, the odd part of $\fp^x$; this is a closed subvariety of $X_0 \times \fg_1$. Let $Z$ be the image of $Y$ in $\fg_1$; explicitly, $Z$ is the union of all $G_0$-conjugates of $\fp_1$. In many cases, $Z$ is a determinantal variety, or something of a similar flavor. Let $\tilde{Z}$ be the affinization of $Y$, which is a finite cover of $Z$. We refer to this ensemble of varieties as the \defn{Grothendieck--Springer theory} for $X$, since it is analogous to the classical Grothendieck--Springer resolution (and contains some instances of it). We emphasize that $X_0$, $Y$, $Z$, and $\tilde{Z}$ are ordinary (not super) varieties.

Combining the above spectral sequence and the geometric method, we obtain the following theorem. It establishes a link between the cohomology of flag supervarieties and syzygies of determinantal-like varieties.

\begin{theorem} \label{mainthm2}
Suppose that $\rH^i(Y, \cO_Y)=0$ for $i>0$. Letting $S=\Sym(\fg_1^*)$, we have a canonical isomorphism
\begin{displaymath}
\rH^q(X, \lw^{p+q}(\cI/\cI^2)) = \Tor_p^S(\cO_{\tilde{Z}}, \bC)_{p+q}
\end{displaymath}
and a spectral sequence
\begin{displaymath}
\rE_1^{p,q} = \Tor_{-q}^S(\cO_{\tilde{Z}}, \bC)_p \implies \rH^{p+q}(X, \cO_X).
\end{displaymath}
\end{theorem}

\subsection{The case of Grassmannians} \label{ss:intro-grass}

Suppose now that $X=\Gr_{r|s}(\bC^{n|m})$ is the super Grassmannian. We apply the method discussed above to study $\rH^*(X, \cO_X)$. Since $\Gr_{r|s}(\bC^{n|m}) \cong \Gr_{s|r}(\bC^{m|n})$, we may assume, without loss of generality, that $r \ge s$. We summarize some of the key points here.

Let $V_0=\bC^n$ and $V_1=\bC^m$, and put
\begin{displaymath}
W_0=\Hom(V_0,V_1), \quad W_1=\Hom(V_1, V_0), \quad W=W_0 \times W_1.
\end{displaymath}
Thus $W$ is identified with the odd part of the Lie superalgebra $\fgl_{n|m}$. The space $Y$ defined in \S \ref{ss:intro-gen} can be described as follows: a point corresponds to a tuple $(f,g,R_0,R_1)$ where
\begin{itemize}
\item $R_0$ is an $r$-dimensional subspace of $V_0$
\item $R_1$ is an $s$-dimensional subspace of $V_1$
\item $f \colon V_0 \to V_1$ is a linear map satisfying $f(R_0) \subset R_1$
\item $g \colon V_1 \to V_0$ is a linear map satisfying $g(R_1) \subset R_0$.
\end{itemize}
The projection map $Y \to W$ takes $(f,g,R_0,R_1)$ to $(f,g)$. The only condition on points in the image is that the nullity of $f$ must be at least $r-s$. Hence $Z = Z_0\times W_1$ where $Z_0$ is the determinantal variety consisting of linear maps with nullity at least $r-s$; if $n-m \ge r-s$, then $Z_0=W_0$ and hence $Z=W$.

It turns out that the affinization $\tilde{Z}$ of $Y$ is typically not $Z$; that is, there are global functions on $Y$ that do not factor through $Z$. To see this, suppose that $(f,g,R_0,R_1)$ is a point of $Y$. Then $fg$ is an operator on $V_1$ that preserves the subspace $R_1$. Thus the characteristic polynomial of $fg \vert_{R_1}$ is a factor of the characteristic polynomial of $fg$, and its coefficients give additional global functions on $Y$. We show that these generate all the additional functions on $Y$; this is a non-trivial theorem. From this, we see that $\tilde{Z}$ can be described as the space of tuples $(f,g,p)$ where $f \in Z_0$, $g \in Z_1$, and $p$ is an appropriate factor of the characteristic polynomial of $fg$.

We prove a number of results about the geometry of this situation. Notably, we show that $\tilde{Z}$ has rational singularities, which allows us to show that $\rH^i(Y, \cO_Y)=0$ for $i>0$. We also show that $\cO_{\tilde{Z}}$ is a free $\cO_Z$-module of finite rank, and that $\cO_{\tilde{Z}} \otimes_{\cO_Z} \bC$ is the singular cohomology ring of a certain Grassmannian. Thus the free resolution of $\cO_{\tilde{Z}}$ over $S=\Sym(W^*)$ can be determined from that of $\cO_Z$, which was explicitly computed by Lascoux \cite{lascoux}. Applying Theorem~\ref{mainthm2}, we obtain a spectral sequence computing $\rH^*(X, \cO_X)$, where the terms are composed of the resolution of $\cO_Z$ and the singular cohomology of a Grassmannian. Using Lascoux's work, we show that this spectral sequence degenerates. The final result is the following theorem:

\begin{theorem} \label{mainthm}
Let $X=\Gr_{r|s}(\bC^{n|m})$ with $r \ge s$. We have the following:
\begin{enumerate}
\setlength{\itemsep}{4pt}
\item Suppose $n-m \ge r-s$. 
\begin{enumerate}[label=(\roman*)]
\item $\rH^*(X, \cO_X)$ is naturally isomorphic to $\rH^*_{\rm sing}(\Gr_s(\bC^m), \bC)$ as a graded algebra.
\item The $\GL_{n|m}$ action on $\rH^*(X, \cO_X)$ is trivial.
\end{enumerate}
\item Suppose $r-s \ge n-m$, and let $A^*=\rH^*_{\rm sing}(\Gr_s(\bC^{n+s-r}), \bC)$.
\begin{enumerate}[label=(\roman*)]
\item We have a natural isomorphism $\rH^*(X, \cO_X)^{\GL_{n|m}}=A^*$ of graded algebras.
\item There is a graded $\GL_{n|m}$-representation $E^*$ such that $\rH^*(X, \cO_X)$ is isomorphic to $A^* \otimes E^*$, as a graded $\GL_{n|m}$-equivariant $A^*$-module.
\item Let $Z_0 \subset \Hom(\bC^n, \bC^m)$ be the determinantal variety consisting of linear maps of rank $\le n+s-r$, and regard $\cO_{Z_0}$ as a quotient of $S=\Sym(\bC^n \otimes (\bC^m)^*)$. Then we have an isomorphism of $\GL_n \times \GL_m$ representations
\begin{displaymath}
E^i = \bigoplus_{p \ge 0} \Tor_p^S(\cO_{Z_0}, \bC)_{i+p},
\end{displaymath}
where the subscript on $\Tor$ denotes the appropriate graded piece. That is, $E^i$ is the $i$th linear strand of the free resolution of $\cO_{Z_0}$.
\end{enumerate}
\end{enumerate}
\end{theorem}

We note that in (b) above, $Z_0 = \Hom(\bC^n, \bC^m)$ if $r=s$ (and $m \ge n$). In particular, $E^i=0$ for $i \ne 0$ and $E^0=\bC$ and so this situation behaves like case (a).

\subsection{Supergroup representations on syzygies}

In 1985, Pragacz and Weyman \cite{pragacz-weyman} discovered that the free resolutions of determinantal ideals carry a natural (but very much non-obvious) representation of a general linear Lie superalgebra. Their construction was generalized and simplified in some subsequent work of Akin--Weyman \cite{akin-weyman, akin-weyman2, akin-weyman3} and Raicu--Weyman \cite{raicu-weyman}. Theorem~\ref{mainthm} leads to a much more direct and conceptual construction of this action. Indeed, by considering the case where $s=0$, we find that the $i$th linear strand of the resolution of a determinantal variety is identified with $\rH^i(X, \cO_X)$ for an appropriate super Grassmannian $X$, and this obvious carries an action of the super general linear group (as it acts on $X$). See \S \ref{ss:pw} for more details.

The first author generalized the results of Pragacz--Weyman to (skew-)symmetric matrices in \cite{sam}. Here the resolution carries an action of the periplectic superalgebra. The methods of this paper also apply to that case. The details will be carried out in \cite{superres2}.

Rather than compute syzygies of determinantal ideals with respect to a polynomial ring, one can consider them with respect to an intermediate quotient ring. Using a certain class of complete intersections as such an intermediate quotient ring, a third class of superalgebra actions was discovered in \cite{sam-osp} using the orthosymplectic Lie superalgebra. These representations are infinite-dimensional and do not appear to fall into the mold of the current paper. However, one can attempt to treat the orthosymplectic Lie superalgebra using the methods of this paper; see the comments in the next section.

\subsection{Further work}

We have used Theorem~\ref{mainthm2} to compute the cohomology of various other flag supervarieties. The details of this work will appear in a subsequent paper \cite{superres2}. We summarize some of the results here:
\begin{itemize}
\item Let $X$ be the partial flag supervariety parametrizing flags $V_1 \subset \cdots \subset V_k$ in $\bC^{n|m}$ where $\dim(V_i)=r_i|s_i$. Suppose that $n \ge m$ and $n-m \ge r_i-s_i$ for each $i$. Then we compute $\rH^*(X, \cO_X)$ completely. The result is similar to Theorem~\ref{mainthm}(a), in that no interesting syzygies appear.
\item Let $V$ be a super vector space of dimension $n|n$ equipped with a periplectic form, and let $X$ be the Grassmannian parametrizing totally isotropic subspaces of $V$ of dimension $r|n-r$. Then we compute $\rH^*(X, \cO_X)$ completely, and the result is analogous to Theorem~\ref{mainthm}. In this case, the cohomology is composed of the singular cohomology of an isotropic Grassmannian for a symplectic group and the syzygies of a determinantal ideal of (skew-)symmetric matrices. This result gives a conceptual explanation of the work of the first author \cite{sam} that these syzygies carry a representation of the periplectic supergroup.
\item Consider an isotropic super Grassmannian $X$ associated to the orthosymplectic supergroup. In this case, the variety $Z$ is a somewhat mysterious analog of a determinantal variety that we have not previously encountered. Our results in this case are more limited, as we do not know the syzygies of $Z$ in general, and there are examples where the spectral sequence computing $\rH^*(X, \cO_X)$ does not degenerate.
\item Let $X$ be a Grassmannian associated to the isomeric (also known as type Q, see \cite[\S 1.5]{isomeric}) supergroup. In this case, $X_0$ is an ordinary Grassmannian and $\cI/\cI^2$ is the sheaf of 1-forms on $X_0$. We thus see that $\rH^*(X, \lw(\cI/\cI^2))$ is the de Rham cohomology of $X_0$ and the map $Y \to Z$ is the usual Grothendieck--Springer map for the general linear Lie algebra. This gives a very direct connection between $\rH^*(X, \cO_X)$ and the singular cohomology of $X_0$. We can also handle partial flag varieties associated to the isomeric supergroup.
\item We can also treat some exceptional supergroups. For example, if $X$ is a complete flag supervariety associated to $D(2,1;\alpha)$ then we compute $\rH^*(X, \cO_X)$ completely. The variety $Z$ in this case is the zero locus of the hyperdeterminant on $(\bC^2)^{\otimes 3}$.
\end{itemize}

There is one other direction coming out of Theorem~\ref{mainthm} that we are currently pursuing. Let $X=\Gr_{r|s}(\bC^{\infty|\infty})$ be the infinite super Grassmannian (still with $r \ge s$), and let $\cR$ and $\cQ$ be the tautological sub and quotient bundles on $X$. A natural problem is to study the cohomology of the equivariant bundle $\cE=\bS_{\lambda}(\cQ) \otimes \bS_{\mu}(\cR)$, where $\bS_{\lambda}$ denotes a Schur functor. Theorem~\ref{mainthm} implies that $\rH^*(X, \cO_X)$ is isomorphic to a polynomial ring $\bC[c_1, \ldots, c_s]$ in this case, where $r \ge s$ and $c_i$ has degree $2i$. We thus see that $\rH^*(X, \cE)$ is a module over this polynomial ring. In \cite{superbwb}, we study $\rH^*(X, \cE)$ from this perspective. The module structure allows us to prove interesting results about the cohomology groups, and offers a kind of explanation for the failure of the classical Borel--Weil--Bott theorem in this setting.

\begin{remark}
Generalizations of Borel--Weil--Bott to flag supervarieties appear in the literature: see, for instance, \cite{coulembier,gruson,penkov,penkov-serganova}. However, the existing work (that we are aware of) is of a quite different flavor than this paper, and generally does not treat the case of the structure sheaf, or the vector bundles considered above, with the exception of super projective space, see \cite{gruson, musson-serganova, serganova}.
\end{remark}

\subsection{Questions and comments}

The results of this paper raise a few issues that deserve further exploration.

\begin{itemize}
\item We show that the spectral sequence in Theorem~\ref{mainthm2} degenerates for the super Grassmannian. Our proof is rather special and relies on the fact that the representations appearing in the Tor group of a determinantal ideal are multiplicity-free. Is this a more conceptual reason for this degeneration that does not rely on such explicit combinatorial calculation? One obstacle is that, as noted above, there are examples for the orthosymplectic group in which the spectral sequence does not degenerate.
  
\item Theorem~\ref{mainthm} shows that the coherent cohomology of a super Grassmannian is related to the singular cohomology of some ordinary Grassmannian. The path between these two objects in our proof is somewhat circuitous. Is there a more direct connection? As noted above, in the case of isomeric Grassmannians, the connection to singular cohomology (through de Rham cohomology) is quite clear.
  
\item Are there other homogeneous bundles on the super Grassmannian whose cohomology can be completely computed using the methods of this paper, such as $\Omega^1$? Each bundle still has a filtration and associated graded, but the problem is getting a good description of the cohomology of the resulting sheaf.
  
\item A morphism from a superscheme $X$ to the super Grassmannian is equivalent to the data provided by the pullback of the tautological sequence. Hence, given a subbundle of a trivial superbundle on $X$, we get a morphism from the coherent cohomology of a super Grassmannian to $\rH^0(X,\cO_X)$. This allows one to define ``coherent Chern classes'' for superbundles.
\end{itemize}

\subsection{Outline}

In \S \ref{s:general}, we develop our general method for studying the cohomology of flag supervarieties. The rest of the paper is devoted to carrying out this method in the case of super Grassmannians. In \S \ref{s:split} and \S \ref{s:prep}, we give some preparatory material needed for this. The heart of the paper is \S \ref{s:gs}, where we study the Grothendieck--Springer theory associated to the super Grassmannian. Finally, in \S \ref{s:coh}, we apply this work to compute the cohomology of the super Grassmannian.

\subsection{Conventions}

Since we treat both superalgebras and ordinary (non-super) algebras, we will reserve the term ``algebra'' or ``ring'' to mean an ordinary structure, and always use the prefix ``super'' when it is used. The same applies to other objects, such as groups and schemes. 

If $X$ is an affine scheme, we use $\cO_X$ to denote both its structure sheaf and also the ring of global sections; the meaning will be clear from context.

\section{Cohomology of flag supervarieties and syzygies} \label{s:general}

In this section we outline a general approach to studying the cohomology of flag supervarieties by connecting them to certain determinantal-like varieties appearing in an analog of Grothendieck--Springer theory.

\subsection{Super geometry}

We begin by recalling some fundamental notions from super geometry. For our purposes, a \defn{(commutative) super algebra} is a $\bZ/2$-graded associative unital $\bC$-algebra $R$ such that $xy=(-1)^{\vert x \vert \vert y \vert} yx$ holds for all homogeneous elements $x,y \in R$, where $\vert x \vert \in \bZ/2$ denotes the degree of $x$. We use the term \defn{ordinary algebra} to describe a super algebra concentrated in degree~0. Let $R$ be a super algebra. We let $J=J_R$ be the ideal of $R$ generated by its degree~1 elements. Every element of $J$ is nilpotent, and the quotient $R/J$ is an ordinary algebra. We let $\Spec(R)$ be the topological space $\Spec(R/J)$ equipped with a sheaf of super algebras derived from localizations of $R$ in the usual manner.

A \defn{super scheme} over $\bC$ is a topological space equipped with a sheaf of super algebras that is locally isomorphic (in a suitable sense) to a space of the form $\Spec(R)$, where $R$ is a super algebra. We use the term \defn{ordinary scheme} to describe a super scheme for which the structure sheaf is concentrated in degree~0; such an object is a scheme in the classical sense. Let $X$ be a super scheme. We let $\cJ=\cJ_X$ be the ideal sheaf generated by the degree~1 elements of $\cO_X$. The quotient $\cO_X/\cJ_X$ is a sheaf of ordinary rings, and $(X, \cO_X/\cJ_X)$ is an ordinary scheme. We call this the \defn{underlying ordinary scheme} and denote it by $X_{\ord}$. We let $\gr^n(\cO_X)=\cJ^n/\cJ^{n+1}$ and $\gr(\cO_X)=\bigoplus_{n \ge 0} \gr^n(\cO_X)$. Each $\gr^n(\cO_X)$ is a quasi-coherent sheaf on $X_{\ord}$ and $\gr(\cO_X)$ is a quasi-coherent sheaf of super algebras.

We define a super scheme to be \defn{smooth} if $X_{\ord}$ is a smooth algebraic variety, $\cJ/\cJ^2$ is a locally free coherent sheaf, and the natural map $\lw(\cJ/\cJ^2) \to \gr(\cO_X)$ is an isomorphism. We sometimes use the term ``supervariety'' when speaking of smooth super schemes. Suppose $X$ is smooth and irreducible. We define its \defn{dimension} to be the pair $d_0 \vert d_1$, where $d_0$ is the dimension of $X_{\ord}$ and $d_1$ is the rank of $\cJ/\cJ^2$. Given a point $x \in X$, the \defn{cotangent superspace} of $x$ is the super vector space $\fm_x/\fm_x^2$, where $\fm_x$ is the maximal ideal at $x$; this is easily seen to have dimension $d_0 \vert d_1$. The \defn{tangent superspace} of $x$ is the dual of the cotangent superspace.

The following proposition, which is immediate from the definitions, is our primary tool for connecting the coherent cohomology of super varieties to ordinary algebraic geometry:

\begin{proposition} \label{prop:ss}
  \addtocounter{equation}{-1}
  \begin{subequations}
Let $X$ be a smooth super scheme. There is a natural $E_1$ spectral sequence
\begin{equation} \label{eq:spectral}
\rE^{p,q}_1 = \rH^{p+q}(X, \lw^p(\cJ/\cJ^2)) \implies \rH^{p+q}(X, \cO_X).
\end{equation}
\end{subequations}
\end{proposition}

\subsection{The geometric method} \label{ss:gm}

Let $X$ be a complex projective variety, and suppose that
\begin{displaymath}
0 \to \xi \to \epsilon \to \eta \to 0
\end{displaymath}
is an exact sequence of locally free coherent sheaves on $X$, with $\epsilon =W^* \otimes \cO_X$ globally free. We make the following definitions:
\begin{itemize}
\item Let $S=\Sym(W^*)$, a polynomial ring, and $W=\Spec(S)$, an affine space.
\item Let $Y=\Spec(\Sym(\eta))$. This is the total space of a vector bundle over $X$, and a closed subvariety of $X \times W$.
\item Let $Z \subset W$ be the image of $Y$ under the projection map $X \times W \to W$. This is a closed subvariety of $W$ as $X$ is projective.
\item Let $\tilde{Z}$ be the affinization of $Y$, i.e., $\tilde{Z}$ is the affine variety with coordinate ring $\cO_{\tilde{Z}}=\Gamma(Y, \cO_Y)$. The Stein factorization shows that the map $Y \to Z$ factors through a finite morphism $\tilde{Z} \to Z$. In particular, $\cO_{\tilde{Z}}$ is a finite $S$-module.
\end{itemize}
The following proposition describes how to determine the terms of the minimal free resolution of $\cO_{\tilde{Z}}$, in certain situations. We remark that $\rH^i(Y, \cO_Y) = \rH^i(X, \Sym(\eta))$.

\begin{proposition} \label{prop:gm1}
Suppose that $\rH^i(Y, \cO_Y)=0$ for $i>0$. Then we have a natural isomorphism of vector spaces
\begin{displaymath}
\Tor^S_p(\cO_{\tilde{Z}}, \bC)_{p+q} = \rH^q(X, \lw^{p+q}{\xi}),
\end{displaymath}
for all $p,q \in \bZ$, where the subscript indicates the appropriate graded piece.
\end{proposition}

\begin{proof}
  \addtocounter{equation}{-1}
  \begin{subequations}
We briefly sketch the proof. Consider the cartesian diagram
\begin{displaymath}
\xymatrix@C=4em{
X \ar[r]^{g'} \ar[d]_{f'} & X \times W \ar[d]^f \\
\Spec(\bC) \ar[r]^g & W }
\end{displaymath}
where $g'$ is the zero section and $g$ is the inclusion of $0$ in $W$. Regard $\cO_Y$ as a coherent sheaf on $X \times W$. Then we have a natural base change isomorphism \stacks{08IB}
\begin{equation} \label{eq:bc}
\rL g^* \rR f_*(\cO_Y) = \rR f'_* \rL (g')^* (\cO_Y).
\end{equation}
By our assumptions, $\rR f_*(\cO_Y)=\cO_{\tilde{Z}}$, and so the left side above computes $\Tor^S_*(\cO_{\tilde{Z}}, \bC)$. Now, we have a Koszul resolution $\cO_W \otimes \lw(\xi) \to \cO_Y$. Since this is a flat resolution, it can be used to compute $\rL (g')^*$. We find that $\rL (g')^*(\cO_Y) = \lw(\xi)$, where $\lw(\xi)$ is the complex with $\lw^k(\xi)$ in cohomological degree $-k$ and all differentials zero. The derived pushforward of this under $f'$ is $\bigoplus_{k \ge 0} \rH^*(X, \lw^k(\xi))$ (ignoring degrees). We refer to \cite[Ch.~5]{weyman} for additional details.
\end{subequations}
\end{proof}

\begin{remark} \label{rmk:Tor-alg}
The functors $\rL g^* \rR f_*$ and $\rR f'_* \rL(g')^*$ are both lax monoidal, and the base change map $\rL g^* \rR f_* \to \rR f'_* \rL(g')^*$ is one of lax monoidal functors. It follows that the two sides of \eqref{eq:bc} are algebras in the derived category of vector spaces, and the isomorphism is one of algebras. From this, it follows that the isomorphism in Proposition~\ref{prop:gm1} is compatible with the multiplicative structures on each side. This compatibility is used at two points in the proof of Theorem~\ref{mainthm}. 
\end{remark}

\subsection{The geometric method in super geometry} \label{ss:super-geom}

Combining Propositions~\ref{prop:ss} and~\ref{prop:gm1}, we obtain the following important result:

\begin{theorem} \label{thm:super-gm}
Let $X$ be a smooth super variety with $X_{\ord}$ projective. Suppose that
\begin{displaymath}
0 \to \cJ/\cJ^2 \to \epsilon \to \eta \to 0
\end{displaymath}
is an exact sequence of locally free coherent sheaves on $X_{\ord}$, with $\epsilon=W^* \otimes \cO_{X_{\ord}}$ globally free. Let $S=\Sym(W^*)$, let $Y=\Spec(\Sym(\eta))$, and let $\tilde{Z}$ be the affinization of $Y$. Suppose that $\rH^i(Y, \cO_Y)=0$ for $i>0$. Then we have a natural isomorphism
\begin{displaymath}
\Tor^S_p(\cO_{\tilde{Z}}, \bC)_{p+q} = \rH^q(X, \lw^{p+q}(\cI/\cI^2)),
\end{displaymath}
and a spectral sequence
\begin{displaymath}
\rE^{p,q}_1 = \Tor^S_{-q}(\cO_{\tilde{Z}}, \bC)_p \implies \rH^{p+q}(X, \cO_X).
\end{displaymath}
\end{theorem}

We note one corollary of the theorem. Define the \defn{super dimension} of a super vector space $V$ by $\sdim(V) = \dim(V_0)-\dim(V_1)$, and define the \defn{super Euler characteristic} of a super variety $X$ by
\begin{displaymath}
\chi(X) = \sum_{i \ge 0} (-1)^i \sdim{\rH^i(X, \cO_X)}.
\end{displaymath}

\begin{corollary} \label{cor:euler}
Suppose we are in the setting of Theorem~\ref{thm:super-gm}. Then $\chi(X)$ is the degree of the map $\tilde{Z} \to W$ over the generic point of $W$.
\end{corollary}

\begin{proof}
From the spectral sequence, we find $\chi(X)=\sum_{p \ge 0} (-1)^p \dim \Tor^S_p(\cO_{\tilde{Z}}, \bC)$, which is easily seen to coincide with the dimension of the vector space $\Frac(S) \otimes_S \cO_{\tilde{Z}}$.
\end{proof}

We note that the corollary implies that $\chi(X)=0$ if $Z$ is a proper subvariety of $W$, i.e., if $Y \to W$ is not surjective. In this case, assuming $\sdim{\rH^0(X, \cO_X)}\ne 0$, we find that $\cO_X$ necessarily has higher cohomology. One can prove a number of variants of the above corollary, e.g., allowing for cases where $\rH^i(Y, \cO_Y)$ is non-vanishing, or replacing ``super dimension'' with ``super character'' in the definition of $\chi(X)$ when there is a group acting.

\subsection{Grothendieck--Springer theory for flag supervarieties} \label{ss:gs}

Let $G$ be a connected reductive supergroup, let $P$ be a parabolic subsupergroup of $G$, and let $X=G/P$ be the associated flag supervariety. We now take a look at what Theorem~\ref{thm:super-gm} amounts to in this case. We note that this discussion is included only to sketch the general situation and provide context; our main results do not depend on it.

Let $G_0$, $P_0$, and $X_0$ be the ordinary varieties underlying $G$, $P$, and $X$. Then $G_0$ is a connected reductive group, $P_0$ is a parabolic subgroup of $G_0$, and $X_0=G_0/P_0$ is a generalized flag variety. Let $x=gP_0$ be a complex point of $X_0$. Let $P^x={}^gP$ and $P^x_0={}^gP_0$, where ${}^g(-)$ denotes conjugation by $g$. In this way, $X_0$ can be seen as parametrizing $G_0$-conjugates of either $P$ or $P_0$. Let $\fg$, $\fp$, and $\fp^x$ be the Lie superalgebras of $G$, $P$, and $P^x$. The tangent superspace to $G/P$ at $x$ is $\fg/\fp^x$, while the tangent space to $G_0/P_0$ at $x$ is $\fg_0/\fp^x_0$. Since $\cJ/\cJ^2$ is the conormal bundle to $X_0 \subset X$, we see that the fiber of $\cJ/\cJ^2$ at $x$ is $(\fg_1/\fp^x_1)^*$. Let $\epsilon=\fg_1^* \otimes \cO_X$ and let $\eta$ be the locally free coherent sheaf on $X_0$ whose fiber at $x$ is $\fp_1^*$. We then have a short exact sequence
\begin{displaymath}
0 \to \cJ/\cJ^2 \to \epsilon \to \eta \to 0.
\end{displaymath}
So far, we have just explained that there is such a sequence for each fiber. However, since the sheaves are all $G_0$-equivariant and $G_0$ acts transitively on $X_0$, it follows that there is such a sequence of the sheaves.

We thus see that we are indeed in the setting of Theorem~\ref{thm:super-gm}. The space $W$ is $\fg_1$. The space $Y$ is the vector bundle over $X_0$ whose fiber at $x$ is $\fp^x_1$. In other words, $Y$ consists of points $(x,v)$ where $x \in X_0$ and $v \in \fp^x_1$. Recall that $Z \subset \fg_1$ is the image of $Y$ under the projection map $X_0 \times W \to W$. From the description of $Y$, we see that $Z=\bigcup_{g \in G_0} {}^g\fp_1$.

The space $\tilde{Z}$ is more subtle. Let $\fh \subset \fp_0$ be a Cartan subalgebra of $\fg_0$ and let $\fW$ be its Weyl group. We have maps
\begin{displaymath}
\fg_1 \to \fg_0 \to \fg_0\cq G_0 \xrightarrow{\cong} \fh\cq\fW,
\end{displaymath}
where the first map is self-bracket (i.e., $v \mapsto [v,v]$), the second map is the GIT quotient map, and the third is the isomorphism coming from the Chevalley restriction theorem. We note that $\fh\cq\fW$ is the spectrum of the invariant ring $\Sym(\fh^*)^{\fW}$. As $Z$ is a subvariety of $\fg_1$, we thus obtain a natural map $Z \to \fh\cq\fW$. We have similar maps
\begin{displaymath}
\fp^x_1 \to \fp^x_0 \to \fh\cq\fW'
\end{displaymath}
where $\fW'$ is the subgroup of $\fW$ associated to $P_0$. Since each element of $Y$ belongs to some $\fp^x_1$, these maps define a map $Y \to \fh\cq\fW'$, and this map factors through $\tilde{Z}$ since the target is affine. This leads to a commutative diagram
\begin{equation} \label{eq:hmodW}
\begin{aligned}
\xymatrix{
\tilde{Z} \ar[r] \ar[d] & \fh\cq\fW' \ar[d] \\
Z \ar[r] & \fh\cq\fW. }
\end{aligned}
\end{equation}
As we will see in \S \ref{s:gs}, this diagram is not cartesian in general. However, in all cases we have considered, a slight modification is cartesian.

We thus have the following general plan for studying $\rH^{\bullet}(X, \cO_X)$:
\begin{enumerate}
\item Show that $\rH^i(Y, \cO_Y)$ vanishes for $i>0$.
\item Determine the minimal free resolution of $\cO_{\tilde{Z}}$ over $S=\Sym(\fg_1^*)$.
\item Analyze the spectral sequence in Theorem~\ref{thm:super-gm}.
\end{enumerate}
In the remainder of this paper, we carry out this plan in detail for $X=\Gr_{r|s}(\bC^{n|m})$.

\section{Splitting rings} \label{s:split}

In \S \ref{s:gs}, we study Grothendieck--Springer theory in detail for the super Grassmannian. In this section and the next section, we gather some preliminary material that will be needed for that. This section develops the theory of splitting and factorization rings. These have also been considered in \cite{ekedahl-laksov, GSS, laksov}.

\subsection{Splitting rings} \label{ss:split}

Let $A$ be a ring and let $f=\sum_{i=0}^n a_{n-i} u^i$ be a monic polynomial with coefficients in $A$ (so $a_0=1$). We define the \defn{splitting ring} of $f$, denoted $\Split_A(f)$, to be the quotient $A[\xi_1, \ldots, \xi_n]/I$ where $I$ is the ideal generated by the elements
\[
  a_i - (-1)^i e_i(\xi_1, \ldots, \xi_n), \qquad 1 \le i \le n,
\]
where $e_i$ is the $i$th elementary symmetric function. Thus $\Split_A(f)$ is the universal quotient of $A[\xi_i]$ in which we have $f(u)=\prod_{i=1}^n(u-\xi_i)$. If $A$ is graded and $a_i$ has degree $2i$ then $\Split_A(f)$ is naturally graded, with $\xi_i$ of degree~2. (We include a factor of~2 here for later convenience.) The symmetric group $\fS_n$ acts on $\Split_A(f)$ by permuting the $\xi_i$'s and fixing $A$. Formation of the splitting ring is compatible with base change: if $A \to A'$ is a homomorphism, and $f'$ is the image of $f$ under $A[u] \to A'[u]$, then we have a natural isomorphism
\[
  \Split_{A'}(f')=A' \otimes_A \Split_A(f).
\]
We fix $A$ and $f$ for the remainder of this section, and let $B=\Split_A(f)$. We also assume $A$ to be noetherian.

\subsection{The universal case}

Let $A^{\univ}=\bZ[a_1, \ldots, a_n]$, let $f^{\univ} \in A^{\univ}[u]$ be the polynomial $u^n+\sum_{i=0}^{n-1} a_{n-i} u^i$, and let $B^{\univ}$ be the splitting ring of $f^{\univ}$. The map $\bZ[\xi_1, \ldots, \xi_n] \to B^{\univ}$ is surjective, and has no kernel since $B^{\univ} \otimes \bC$ clearly has Krull dimension $n$. Thus $B^{\univ}=\bZ[\xi_1, \ldots, \xi_n]$. 

The polynomial $f \in A[u]$ is the image of $f^{\univ}$ under a unique ring homomorphism $A^{\univ} \to A$. Since formation of splitting rings is compatible with base change, we have $B=A \otimes_{A^{\univ}} B^{\univ}$. This can be a useful tool for proving results about general splitting rings.

The above discussion shows that we have a cartesian diagram
\begin{displaymath}
\xymatrix{
\Spec(B) \ar[r] \ar[d] & \Spec(B^{\univ}) \ar[d] \\
\Spec(A) \ar[r] & \Spec(A^{\univ}) }
\end{displaymath}
We can identify $\Spec(B^{\univ})$ with $\bA^n$ and $\Spec(A^{\univ})$ with the quotient $\bA^n \cq \fS_n$ (which is isomorphic to $\bA^n$). Here, $\bA^n$ and $\fS_n$ can be identified with the Cartan subalgebra and Weyl group of $\gl_n$. The above diagram is thus similar to \eqref{eq:hmodW} (with $P_0$ a Borel, so that $\fW'=1$), and this is essentially how splitting rings will be relevant in \S \ref{s:gs}.

\subsection{Basic results}

Let $\Delta \in A$ be the discriminant of $f$. We have $\Delta=\prod_{i \ne j} (\xi_i-\xi_j)$ in $B$. Recall that a ring homomorphism is \defn{syntomic} if it is flat, of finite presentation, and all of the fibers are local complete intersections \stacks{00SL}.

\begin{proposition} \label{prop:split}
We have the following:
\begin{enumerate}
\item As an $A$-module, $B$ is free of rank $n!$.
\item The map $A \to B$ is a syntomic.
\item If $A$ satisfies Serre's condition $(S_k)$, then so does $B$. In particular, if $A$ is Cohen--Macaulay, then so is $B$.
\item If $\Delta$ is a unit of $A$ then $A \to B$ is \'etale.
\item If $A$ is reduced and $\Delta$ is a non-zerodivisor then $B$ is reduced.
\item The inclusion $A \to B$ admits an $A$-linear splitting.
\end{enumerate}
\end{proposition}

\begin{proof}
  (a) It suffices, by base change, to prove the statement in the universal case where $A=\bZ[a_1, \ldots, a_n]$ and $f=u^n+\sum_{i=0}^{n-1} a_{n-i} u^i$. This is well-known, and there are several proofs, but we include one. The map $A \to B$ is finite, as each $\xi_i$ satisfies the universal polynomial. Since $A \to B$ is a finite map of polynomial rings of the same dimension, it is flat \stacks{00R4}. Therefore $B$ is projective as an $A$-module, and thus free, as any projective $A$-module is free (this statement is easy in this case as everything is graded). The rank can be computed over the generic point, where it is well-known to be $n!$.

(b) Suppose that $\fp$ is a prime of $A$. Then $B \otimes_A \kappa(\fp)$ is finite over $\kappa(\fp)$ by~(a), and therefore of Krull dimension~0. This ring is a quotient of $\kappa(\fp)[\xi_1,\dots,\xi_n]$ by $n$ relations, and is therefore a complete intersection. It follows that $A \to B$ is syntomic.

(c) Since $(S_k)$ is syntomic local \stacks{036A}, the result follows from (b).

(d) We have $0=f(\xi_i)$ and so $0=f'(\xi_i) d\xi_i$. However, $f'(\xi_i)=\prod_{j \ne i} (\xi_i-\xi_j)$ divides $\Delta$ and is therefore a unit. Thus $d\xi_i=0$. We conclude that $\Omega_{B/A}=0$. Since $B$ is finite flat over $A$ by (a), it is therefore \'etale.

(e) Since $A$ is reduced, it satisfies $(R_0)$ and $(S_1)$. Thus $B$ satisfies $(S_1)$ by part~(c). Since $V(\Delta) \subset \Spec(A)$ has codimension~1 and $A[1/\Delta] \to B[1/\Delta]$ is \'etale, it follows that $B$ satisfies $(R_0)$. Thus $B$ is reduced.

(f) This follows from the existence of the Gysin homomorphism constructed in \cite[Theorem 8.1]{laksov} (it has to be precomposed with multiplication by $\prod_i \xi_i^{n-i}$ which is also $A$-linear).
\end{proof}

\begin{proposition} \label{prop:split-rep}
Suppose that $n!$ is invertible in $A$. Then $B$ is free of rank one as an $A[\fS_n]$-module.
\end{proposition}

\begin{proof}
Let $\Lambda=\bZ[1/n!]$ and let $S^{\lambda}$ be the Specht module for $\fS_n$ over $\Lambda$. If $M$ is any $\Lambda[\fS_n]$-module, then the natural map
\begin{displaymath}
\bigoplus_{\vert \lambda \vert=n} \Hom_{\Lambda[\fS_n]}(S^{\lambda}, M) \otimes S^{\lambda} \to M
\end{displaymath}
is an isomorphism. To see this, note that $S^{\lambda}$ is a projective $\Lambda[\fS_n]$-module: Young idempotents are defined over $\Lambda[\fS_n]$ (see \cite[Lemma 4.26]{fultonharris}), so $S^\lambda$ is a direct summand of $\Lambda[\fS_n]$. Hence both sides above are exact functors of $M$. It thus suffices to prove the result for $M=S^{\mu}$ (since these are enough projectives), where it is clear.

Now, it suffices to prove the proposition in the universal case $A=\Lambda[a_1, \ldots, a_n]$. By the above isomorphism, we see that $N_{\lambda}=\Hom_{\Lambda[\fS_n]}(S^{\lambda}, B)$ is a summand of $B$, and thus projective as an $A$-module; since $N_{\lambda}$ is also graded, it is free. To prove the result, it is enough to show that the $A$-rank of $N_{\lambda}$ coincides with the $\Lambda$-rank of $S^{\lambda}$. This can be checked over $\Frac(A)$, where it is well known: $\bQ(\xi_1, \ldots, \xi_n)$ is isomorphic to the regular representation of $\fS_n$ over $\bQ(a_1, \ldots, a_n) = \bQ(\xi_1,\dots, \xi_n)^{\fS_n}$ by the normal basis theorem in Galois theory.
\end{proof}

\begin{remark}
The hypothesis that $n!$ is invertible is necessary: indeed, if $n=2$, $A=\bF_2$, and $f=u^2$ then $B \cong \bF_2[t]/t^2$ with trivial $\fS_2$-action, but $\bF_2[\fS_2]$ has non-trivial $\fS_2$-action.
\end{remark}

We pause to give a geometric source of splitting rings (see \cite[Theorem 6.1]{GSS}). Let $X$ be a smooth variety over an algebraically closed field and let $\cE$ be a rank $n$ vector bundle on $X$. Let $A$ be the Chow ring of $X$ and let $a_i = (-1)^i c_i(\cE)$, where $c_i(\cE)$ is the $i$th Chern class of $\cE$. Then the splitting ring $B$ is the Chow ring of the relative flag variety $\Fl(\cE)$. Furthermore, on $\Fl(\cE)$, the pullback of $\cE$ has a complete flag of subbundles (i.e., whose successive quotients are line bundles), and the Chern classes of these line bundles are identified with the $-\xi_i$. 

An important case for us is when $X = \Spec(\bC)$ and $\cE=\bC^n$, so that $f=u^n$. In that case, this discussion gives the following result (we note that the Chow ring and singular cohomology ring of $\Fl(\bC^n)$ are isomorphic since it has a cellular decomposition).

\begin{proposition} \label{prop:split-flag}
Suppose that $f=u^n$. Regard $A$ as graded and concentrated in degree~$0$, and $B$ as graded with each $\xi_i$ of degree~$2$. Then we have a natural isomorphism of graded rings
\begin{displaymath}
B = A \otimes \rH^*_{\sing}(\Fl(\bC^n), \bZ).
\end{displaymath}
\end{proposition}

\subsection{First normality criterion} \label{ss:norm}

We now turn our attention to the question of when $A$ is normal. We will give an initial criterion here, and a variant in \S \ref{ss:norm2} that is somewhat more convenient. Let $\tilde{E} \subset \Spec(B)$ be the closed set where at least three of the $\xi_i$'s coincide or there are two pairs of $\xi_i$ which coincide, and let $E \subset \Spec(A)$ be the image of $\tilde{E}$. Note that $E$ is closed since $A \to B$ is finite. We say that an element $f$ of a normal ring $R$ is \defn{squarefree} if $v_{\fp}(f) \in \{0,1\}$ for all height one primes $\fp$ of $R$, where $v_{\fp}$ denotes the valuation associated to $\fp$. We say that a subset of $\Spec(A)$ has \defn{codimension $\ge c$} if all primes it contains have height $\ge c$.

\begin{proposition} \label{prop:norm-crit}
Suppose the following conditions hold:
\begin{enumerate}
\item $A$ is normal,
\item $\Delta$ is squarefree and a non-zerodivisor,
\item $E$ has codimension $\ge 2$ in $\Spec(A)$.
\end{enumerate}
Then $B$ is normal.
\end{proposition}

\begin{proof}
  First suppose that $A$ is a strictly henselian discrete valuation ring. We show that $B$ is regular. If $\Delta$ is a unit of $A$ then $B$ is \'etale over $A$ and thus regular. Assume then that $\Delta$ is not a unit; by hypothesis (b), it is a uniformizer of $A$. Let $\ol{f}$ be the image of $f$ in $\kappa[u]$, where $\kappa$ is the residue field of $A$. Since $\Delta$ maps to~0 in $\kappa$ it follows that $\ol{f}$ has a repeated root; by hypothesis (c), it has only one (i.e., it has $n-1$ distinct roots). We thus have a factorization $\ol{f}=\ol{q} \cdot \ol{g}$ over $\kappa$, where $\ol{q}$ is a degree~2 polynomial with a repeated root over $\ol{\kappa}$ and $\ol{g}$ has distinct roots over $\ol{\kappa}$. Since $\kappa$ is separably closed, it follows that $\ol{g} = (u-\ol{x}_3) \cdots (u-\ol{x}_n)$ splits completely for $\ol{x}_3,\dots,\ol{x}_n \in \kappa$ distinct (and distinct from the unique root of $\ol{q}$). By the henselian property, we thus have a factorization $f(u)=q(u) (u-x_3) \cdots (u-x_n)$, where $x_i \in A$ lifts $\ol{x}_i$ for $i \ge 3$ and $q(u)$ is a quadratic polynomial with coefficients in $A$.

Now let $\fp$ be a prime of $B$ above the maximal ideal of $A$, and work in $B_{\fp}$ in what follows. Applying a permutation if necessary, we can assume that $\ol{\xi}_i=\ol{x}_i$ for $i \ge 3$, where $\ol{\xi}_i$ is the image of $\xi_i$ in $B_\fp/\fp \cong \kappa$. For $i \ge 3$, it follows that $q(\xi_i)$ and $\xi_i-x_j$, for $j \ne i$, are non-zero in $\kappa$, and thus units of $B_\fp$; since $f(\xi_i)=0$, we conclude that $\xi_i=x_i$. This shows that $B_{\fp}$ is generated as an $A$-algebra by $\xi_1$ and $\xi_2$. We have
\[
  q(u) \prod_{i \ge 3} (u-\xi_i)=\prod_{i \ge 1} (u-\xi_i).
\]
Since monic polynomials are non-zerodivisors, it follows that $q(u) = (u-\xi_1)(u-\xi_2)$. From this, we see that $B_{\fp} \cong A[u]/q$.

We consider two cases. If $\ol{q}$ is irreducible (this is only possible if $\kappa$ has characteristic $2$), then $B_\fp/\fm B_\fp$ is a degree $2$ field extension of $\kappa$ and hence $\fm B_\fp$ is the maximal ideal of $B_\fp$ which is generated by a single non-nilpotent element, so $B_\fp$ is a DVR.

Otherwise, we may assume that $\ol{q} = (u-\ol{x}_1)^2$ for some $\ol{x}_1 \in \kappa$. Let $\pi$ be the discriminant of $q$. This is an element of $A$ that divides $\Delta$ and is~0 in $\kappa$, and is therefore a uniformizer. Up to a linear substitution in $A$, we may assume that $\ol{q} = u^2$. Write $q = u^2 + \alpha_1 u + \alpha_2$. Then $\alpha_1,\alpha_2 \in \fm$ but $\alpha_2 \notin \fm^2$ since $\pi = \alpha_1^2 - 4\alpha_2$ is a uniformizer, and hence $\alpha_2$ is also a uniformizer for $A$. The maximal ideal of $B_\fp$ is generated by $u$ and $\fm B_\fp$, but the latter is generated by $\alpha_2 = -(u+\alpha_1)u$, and so just $u$ suffices to generate. Finally, $u$ is not nilpotent since $\alpha_2$ is not, which shows that $B_\fp$ is a DVR.

We now treat the general situation. Let $\fp$ be a height one prime of $A$. We show that $B_{\fp}$ is regular. Let $A_{\fp}^{\rm sh}$ be the strict henselization of the DVR $A_{\fp}$. By the previous paragraphs, we see that $B_{\fp} \otimes_{A_{\fp}} A_{\fp}^{\rm sh}$ is regular. Now, $A_{\fp}^{\rm sh}$ is the direct limit of a family $\{A_i\}$ of rings, each of which is an \'etale cover of $A_{\fp}$. The above arguments apply with $A_i$ in place of $A_{\fp}^{\rm sh}$ for $i$ sufficiently large. We conclude that $B_{\fp} \otimes_{A_{\fp}} A_i$ is regular for some $i$. Since regularity is \'etale local, we conclude that $B_{\fp}$ is regular. Thus $B$ is regular in codimension~1. Finally, since $A$ is normal, it satisfies Serre's condition $(S_2)$, and hence the same is true for $B$ by Proposition~\ref{prop:split}(c), so $B$ is therefore normal.
\end{proof}

\subsection{Preliminaries on discriminants}

In the following subsection, we refine the above normality criterion. Here we prove a few lemmas that will be needed to do this.

\begin{lemma} \label{lem:disc-1}
Let $S$ be a local ring with maximal ideal $\fm$, and let $X$ be an $n \times n$ matrix with entries in $S$. Suppose that the reduction of $X$ modulo $\fm$ has nullity $\ge r$. Then $\det(X) \in \fm^r$.
\end{lemma}

\begin{proof}
We argue by induction on $n$. For $n=0$ the result is vacuous. Suppose now the result is known for matrices of size $<n$ and let us prove it for matrices of size $n$.
  
Suppose that the first column of $X$ has entries in $\fm$. Let $X_i$ be the matrix obtained from $X$ by deleting the first column and the $i$th row. Then each $X_i$ has nullity $\ge r-1$ modulo $\fm$. Thus, by induction, $\det(X_i) \in \fm^{r-1}$ for each $i$. Taking the Laplace expansion for $\det(X)$ along the first column, we see that $\det(X) \in \fm^r$.

Now suppose that some entry in the first column of $X$ is a unit. Performing row operations, we can reduce to the case where the first column is the first standard basis vector. Let $X'$ be the matrix obtained by deleting the first row and column of $X$. Then $\det(X')=\det(X)$. But $X'$ has nullity $\ge r$ modulo $\fm$, so $\det(X') \in \fm^r$ by induction.
\end{proof}

We recall some basic facts about Sylvester matrices and discriminants. Let $S$ be a commutative ring and let $f(u) = a_0 x^n + \cdots + a_n$ and $g(u) = b_0 x^m + \cdots + b_m$ be univariate polynomials with coefficients in $S$ such that $a_0 \ne 0$ (we do not require anything about $g$ and in fact allow the case that it is identically $0$). We define their \defn{Sylvester matrix} to be the square matrix of size $n+m$ as follows:
\[
  \Syl_{n,m}(f,g) =
  \begin{bmatrix} a_0 & a_{1} & a_{2} & \cdots & a_n & 0 & 0 & \cdots & 0 \\
    0 & a_0 & a_{1} & \cdots & & a_n & 0 & \cdots & 0\\
    \vdots\\
    0&0&0&&&&&\cdots & a_n\\
    b_0 & b_{1} & b_{2} & \cdots\\
    \vdots\\
    0&0&0&&&&&\cdots & b_m
  \end{bmatrix}.
\]
In other words, the first $m$ rows consist of shifts of the sequence $a_0,\dots,a_n$ and the last $n$ rows consist of shifts of the sequence $b_0,\dots,b_m$.

The following is well-known, but we include proofs to keep the discussion self-contained.

\begin{proposition} \label{prop:sylvester}
  \begin{enumerate}
  \item  If $S$ is a field, then 
    \[
      \deg(\gcd(f,g)) = \dim \ker \Syl_{n,m}(f,g).
    \]
  \item If $S$ is graded such that $\deg(a_i)=\deg(b_i)=i$, then $\det(\Syl_{n,m}(f,g))$ is homogeneous of degree $mn$.
  \item If $n = \deg(f)$, then the discriminant of $f$ is $\det(\Syl_{n,n-1}(f,f'))$ (up to a sign) where $f'$ is the derivative of $f$ with respect to $x$.
  \end{enumerate}
\end{proposition}

\begin{proof}
  (a) Let $\alpha = \alpha_{0} u^{m-1} + \cdots + \alpha_{m-1}$ and $\beta = \beta_{0} u^{n-1} + \cdots + \beta_{n-1}$. The coefficients of $\alpha f + \beta g$ are the entries of $\begin{bmatrix} \alpha_{0} & \cdots & \alpha_{m-1} & \beta_{0} & \cdots & \beta_{n-1} \end{bmatrix} \Syl_{n,m}(f,g)$, so that $\ker \Syl_{n,m}(f,g)^T$ is the space of pairs $(\alpha, \beta)$ with $\deg \alpha < m$ and $\deg \beta < n$ such that $\alpha f + \beta g = 0$. Let $h = \gcd(f,g)$ and $f_0=f/h$ and $g_0=g/h$. Then we get $\alpha f_0 = -\beta g_0$ which implies that there is a polynomial $\gamma$ such that $\alpha = g_0 \gamma$ and $\beta = f_0 \gamma$ and
  \[
    \deg \gamma = \deg \beta - \deg f_0 < \deg h.
  \]
  On the other hand, if $\gamma$ is any polynomial with $\deg \gamma < \deg h$, then $\alpha = g_0 \gamma$ has degree $< \deg g \le m$ and $\beta = f_0 \gamma$ has degree $< \deg f = n$, and so we have an isomorphism between the polynomials of degree $< \deg h$ and $\ker \Syl_{n,m}(f,g)$.

  (b) Scaling each $a_i$ and $b_i$ by $\lambda^i$ in $\Syl_{n,m}(f,g)$ is the same as scaling the $i$th column by $\lambda^{i-1}$, the first $m$ rows by $1,\lambda^{-1},\dots,\lambda^{-(m-1)}$, and the last $n$ rows by $1,\lambda^{-1},\dots,\lambda^{-(n-1)}$. Hence we see that $\det(\Syl_{n,m}(f,g))$ is homogeneous of degree $\binom{n+m}{2} - \binom{n}{2} - \binom{m}{2} = mn$.
  
  (c) It suffices to prove this in the universal case $S = \bZ[a_0,\dots,a_n]$ since both the discriminant of $f$ and $\det(\Syl_{n,n-1}(f,f'))$ are polynomials in the coefficients of $f$. We grade $S$ by $\deg(a_i) = i$. In this case, both expressions are homogeneous polynomials of degree $n(n-1)$, this holds for $\det(\Syl_{n,n-1}(f,f'))$ by (b). The discriminant is an irreducible polynomial and vanishes whenever $f$ has a multiple root. By (a), $\det(\Syl_{n,n-1}(f,f'))$ vanishes whenever $\gcd(f,f')$ has positive degree, which is equivalent to $f$ having a multiple root. Hence the two polynomials agree up to sign. 
\end{proof}

\begin{lemma} \label{lem:disc-3}
Suppose $S$ is a local ring with maximal ideal $\fm$ and let $g \in S[u]$ be a monic polynomial of degree $n$. Suppose that the reduction of $g$ modulo $\fm$ has at most $r$ distinct roots in the algebraic closure of $S/\fm$. Then the discriminant of $g$ belongs to $\fm^{n-r}$.
\end{lemma}

\begin{proof}
Let $f$ be the reduction of $g$ modulo $\fm$. By Lemma~\ref{lem:disc-1} and Proposition~\ref{prop:sylvester}, the discriminant of $g$ belongs to $\fm^d$ where $d = \deg(\gcd(f,f'))$. For $\lambda$ in the algebraic closure of $S/\fm$, if $(u-\lambda)^e$ divides $f$, then $(u-\lambda)^{e-1}$ divides $f'$, and so $d \ge n-r$.
\end{proof}

\subsection{Second normality criterion} \label{ss:norm2}

We now give a variant of Proposition~\ref{prop:norm-crit}. Define $V(\Delta, \partial \Delta) \subset \Spec(A)$ to be the set of points $x$ at which $\Delta$ vanishes to order two, in the sense that it belongs to $\fm_x^2$. If $A$ is finitely generated over a field $k$ and $x$ is a smooth point of $\Spec(A)$ then $x$ belongs to $V(\Delta, \partial \Delta)$ if and only if $\Delta=0$ in $\kappa(x)$ and $d\Delta=0$ in $\Omega^1_{A/k} \otimes_A \kappa(x)$; since $\Omega^1_{A/k}$ is locally free on the smooth locus, this shows that $V(\Delta, \partial \Delta)$ is closed in the smooth locus. Recall the set $E \subset \Spec(A)$ from \S \ref{ss:norm}.

\begin{lemma} \label{lem:E-partialDelta}
We have $E \subset V(\Delta, \partial \Delta)$.
\end{lemma}

\begin{proof}
Let $x \in E$, and let $\fp=\fp_x$ be the corresponding prime ideal of $A$. By the definition of $E$, the polynomial $f \in A_{\fp}[u]$ has at most $\deg(f)-2$ distinct roots in the residue field. Thus by Lemma~\ref{lem:disc-3}, we see that $\Delta \in \fp^2 A_{\fp}=\fm_x^2$, and so $x \in V(\Delta, \partial \Delta)$.
\end{proof}

\begin{proposition} \label{prop:norm-crit2}
Suppose the following conditions hold:
\begin{enumerate}
\item $A$ is normal,
\item $\Delta$ is a non-zerodivisor,
\item $V(\Delta, \partial \Delta)$ has codimension $\ge 2$.
\end{enumerate}
Then $\Delta$ is squarefree and $B$ is normal.
\end{proposition}

\begin{proof}
We apply Proposition~\ref{prop:norm-crit}. The set $E$ there has codimension $\ge 2$ by the present assumption (c) and Lemma~\ref{lem:E-partialDelta}. It thus suffices to prove that $\Delta$ is squarefree.

Let $\fp$ be a height one prime of $A$. If $V(\fp)$ is not an irreducible component of $V(\Delta)$ then $v_{\fp}(\Delta)=0$. Suppose now that $V(\fp)$ is an irreducible component of $V(\Delta)$. Then $v_{\fp}(\Delta) \ge 1$, and we show that equality holds. Suppose by way of contradiction that $v_{\fp}(\Delta) \ge 2$. We can then write $\Delta=ab^k$ where $k \ge 2$, $a$ is a rational function on $\Spec(A)$ that is a unit at $\fp$, and $b$ is a rational function on $\Spec(A)$ that is a uniformizer at $\fp$. Since $a$ is a unit of $A_{\fp}$ and $b$ belongs to $\fp A_{\fp}$, we see that $\Delta$ belongs to $\fp^2 A_{\fp}$. Thus $\fp \in V(\Delta, \partial \Delta)$, which contradicts (c). We conclude that $\Delta$ is squarefree.
\end{proof}

\subsection{Factorization rings} \label{ss:factor}

Let $p$ and $q$ be non-negative integers such that $p+q=n$, and put $g=\sum_{i=0}^p b_{p-i} u^i$ and $h=\sum_{i=0}^q c_{q-i} u^i$, where $b_0=c_0=1$ and the remaining $b_i$ and $c_i$ are formal symbols. We define the \defn{$(p,q)$-factorization ring} of $f$, denoted $\Fact^{p,q}_A(f)$ to be $A[b_1, \ldots, b_p, c_1, \ldots, c_q]/I$, where $I$ is the ideal generated by the elements
\[
  a_{n-i} - \sum_{j=0}^i b_{p-j} c_{q-i+j} \qquad 0 \le i \le n-1.
\]
Thus $\Fact^{p,q}_A(f)$ is the universal quotient of $A[b_i,c_j]$ in which we have $f(u)=g(u) h(u)$. It follows from the above formula and the condition $b_0=1$ that $\Fact^{p,q}_A(f)$ is generated as an $A$-algebra by $b_1, \ldots, b_p$. If $A$ is graded and $a_i$ is homogeneous of degree $2i$ then $\Fact^{p,q}_A(f)$ is graded and $b_i$ and $c_i$ have degree $2i$. Formation of the factorization ring is compatible with base change, as with the splitting ring.

In what follows, we let $B=\Split_A(f)$ and $C=\Fact^{p,q}_A(f)$.

\begin{proposition} \label{prop:fact}
We have the following:
\begin{enumerate}
\item We have a natural $A$-algebra isomorphism $B=\Split_C(g) \otimes_C \Split_C(h)$.
\item As an $A$-module, $C$ is free of rank $\binom{n}{p}$.
\item The map $A \to C$ is syntomic. 
\item If $A$ satisfies Serre's condition $(S_k)$, then so does $C$. In particular, if $A$ is Cohen--Macaulay, then so is $C$.
\item If $B$ is reduced (resp., integral, normal), then $C$ is reduced (resp., integral, normal).
\end{enumerate}
\end{proposition}

\begin{proof}
(a) Let $\eta_1, \ldots, \eta_p$ be the generators of $\Split_C(g)$ and $\eta_{p+1}, \ldots, \eta_{p+q}$ those for $\Split_C(h)$. Put $B'=\Split_C(g) \otimes_C \Split_C(h)$. Since $f(u)=\prod_{i=1}^n (u-\eta_i)$ holds over $B'$, we have an $A$-algebra homomorphism $\phi \colon B \to B'$ given by $\phi(\xi_i)=\eta_i$. Let $g^*(u)=\prod_{i=1}^p(u-\xi_i)$ and $h^*(u)=\prod_{i=p+1}^n(u-\xi_i)$ be polynomials in $B[u]$. The factorization $f(u)=g^*(u)h^*(u)$ gives an $A$-algebra homomorphism $C \to B$ mapping $g(u)$ to $g^*(u)$ and $h(u)$ to $h^*(u)$. The tautological splittings of $g^*(u)$ and $h^*(u)$ over $B$ yield an $A$-algebra homomorphism $\psi \colon B' \to B$ given by $\psi(\eta_i)=\xi_i$. Since $\phi$ and $\psi$ are clearly inverses, the result follows.

(b) By Proposition~\ref{prop:split}(a), we have an $A$-module isomorphism $B \cong A^{\oplus n!}$ and $C$-module isomorphisms $\Split_C(g) \cong C^{\oplus p!}$ and $\Split_C(h) \cong C^{\oplus q!}$. Comparing with (a), we obtain an $A$-module isomorphism $C^{\oplus p! q!} \cong A^{\oplus n!}$. It follows that $C$ is projective as an $A$-module of constant rank $\binom{n}{p}$. As in the proof of Proposition~\ref{prop:split}(a), it follows that $C$ is free in the universal case, and thus in all cases.

(c) Let $\fp$ be a prime of $A$. Then $C \otimes_A \kappa(\fp)$ is finite over $\kappa(\fp)$ by (b), and therefore of Krull dimension~0. This ring is a quotient of $\kappa(\fp)[b_1,\dots,b_p,c_1,\dots,c_q]$ by $p+q$ relations, and is therefore a complete intersection. Thus $A \to C$ is syntomic.

(d) This follows since the property is syntomic local.

(e) From (a), $C$ is isomorphic to a subring of $B$, which handles the reduced and integral conditions. For the normality condition, we use that $B$ is an iterated splitting ring over $C$ by (a), and hence the inclusion $C \to B$ admits a $C$-linear splitting by Proposition~\ref{prop:split}(f).
\end{proof}

\begin{remark}
Using Proposition~\ref{prop:split-rep} and Proposition~\ref{prop:fact}(a), we can show that $C=B^{\fS_p \times \fS_q}$ if $n!$ is invertible in $A$. This gives a $C$-linear splitting of $C \to B$ and an alternative proof of Proposition~\ref{prop:fact}(e). In fact, the assumption that $n!$ is invertible is relaxed substantially in \cite[Theorem 3.1]{ekedahl-laksov}: it suffices to know that either $2$ or the discriminant of $f$ is a non-zerodivisor in $A$ to show that $C=B^{\fS_p\times \fS_q}$.
\end{remark}

As with splitting rings, there is a geometric source of factorization rings (see \cite[Theorem 6.1]{GSS}). Let $X$ be a smooth variety over an algebraically closed field and let $\cE$ be a rank $n$ vector bundle on $X$.  Then the $(p,n-p)$ factorization ring $C$ (again taking $A$ to be the Chow ring of $X$ and $a_i = (-1)^i c_i(\cE)$) is the Chow ring of the relative Grassmannian $\Gr(p,\cE)$. When $X = \Spec(\bC)$, we get the following result.

\begin{proposition} \label{prop:fact-grass}
Suppose that $f=u^n$. Regard $A$ as graded and concentrated in degree~$0$, and $C$ as graded in the usual manner. Then we have a natural isomorphism of graded rings
\begin{displaymath}
C = A \otimes \rH^*_{\sing}(\Gr_p(\bC^n), \bZ).
\end{displaymath}
\end{proposition}

\begin{remark}
One can also form partial splitting rings which are intermediate between $B$ and $C$, and all of the above properties generalize. For these rings, one sees the cohomology of a partial flag variety in the analog of Proposition~\ref{prop:fact-grass}.
\end{remark}

\section{Some additional preparatory material} \label{s:prep}

In this section, we give a bit more material that will be needed in \S \ref{s:gs}.

\subsection{Rational singularities} \label{ss:rat-sing}

We first recall some generalities on rational singularities. Let $X$ be an irreducible variety over the complex numbers. We say that $X$ has \defn{rational singularities} if there exists a proper birational map $\pi \colon Y \to X$ with $Y$ smooth (i.e., a resolution of singularities) such that $\pi_*(\cO_Y)=\cO_X$ and $\rR^i \pi_*(\cO_Y)=0$ for $i>0$. If $X$ has rational singularities then for any resolution of singularities $\pi' \colon Y' \to X$ we have $\pi'_*(\cO_{Y'})=\cO_X$ and $\rR^i \pi'_*(\cO_{Y'})=0$ for $i>0$; furthermore, $X$ is normal and Cohen--Macaulay \cite[Theorem 5.10]{kollarmori}. We require the following two additional results concerning rational singularities:
 
\begin{proposition} \label{prop:rat-sing}
  Suppose that $X$ is normal and Cohen--Macaulay and let $\pi \colon Y \to X$ be a resolution of singularities. Let $U \subset X$ be a subset which has rational singularities and suppose that $\pi^{-1}(X \setminus U)$ has codimension $\ge 2$ in $Y$. Then $X$ has rational singularities.
\end{proposition}

\begin{proof}
  Since $X$ is Cohen--Macaulay, let $\omega_X$ be its dualizing sheaf. We have a natural map $\pi_* \omega_Y \to \omega_X$ obtained by duality from $\cO_X \to \pi_* \cO_Y$, and $X$ has rational singularities if and only if this map is surjective \cite[Theorem 5.10]{kollarmori}. This property is local, so we may assume that $X$ is affine. Consider the following commutative diagram:
\begin{displaymath}
\xymatrix{
\Gamma(Y, \omega_Y) \ar[r] \ar[d] & \Gamma(X, \omega_X) \ar[d] \\
\Gamma(\pi^{-1}(U), \omega_Y) \ar[r] & \Gamma(U, \omega_X) }
\end{displaymath}
The bottom map is an isomorphism since $U$ has rational singularities. The vertical maps are isomorphisms since the complements of $U$ and $\pi^{-1}(U)$ have codimension $\ge 2$ (for example, by \cite[Proposition 1.11]{hartshorne}). We conclude that the top map is an isomorphism, as required.
\end{proof}

\begin{proposition} \label{prop:rat-fiber} 
Let $\pi \colon Y \to X$ be a proper morphism of varieties, where $Y$ is smooth and $X$ has rational singularities. Suppose that there is an open dense subset $U$ of $X$ such that $\pi^{-1}(U)$ is isomorphic (as a variety over $U$) to $U \times W$ for some irreducible projective variety $W$ with rational singularities and satisfying $\rH^i(W, \cO_W)=0$ for $i>0$. Then $\rR \pi_*(\cO_Y)=\cO_X$.
\end{proposition}

\begin{proof}
Let $Z \subset Y \times W = Y \times_X (X \times W)$ be the closure of the graph of the isomorphism $\pi^{-1}(U) \to U \times W$, and let $\tilde{Z} \to Z$ be a resolution of singularities. Consider the diagram
\begin{displaymath}
\xymatrix@C=4em{
\tilde{Z} \ar[r]^-{\pi'} \ar[d]_{\rho'} & X \times W \ar[d]^{\rho} \\
Y \ar[r]^-{\pi} & X }
\end{displaymath}
Since $Z$ is the closure of the graph of an isomorphism of open sets, the projection maps $Z \to Y$ and $Z \to X \times W$ are birational. Thus $\pi'$ and $\rho'$ are birational; also, all maps above are proper. We now have
\begin{displaymath}
\rR \pi_*(\cO_Y) = \rR\pi_*(\rR \rho'_*(\cO_{\tilde{Z}})) = \rR \rho_*(\rR \pi'_*(\cO_{\tilde{Z}})) = \rR \rho_*(\cO_{X \times W}) = \cO_X.
\end{displaymath}
In the first step we used that $\rR\rho'_*(\cO_{\tilde{Z}})=\cO_Y$ since $\rho'$ is a proper birational map of smooth varieties; in the second step, we used the commutativity of the diagram; in the third step we used that $\rR\pi'_*(\cO_{\tilde{Z}})=\cO_{X \times W}$ since $\pi'$ is a resolution of $X \times W$, which has rational singularities; and in the final step, we used the assumption on the cohomology of $\cO_W$.
\end{proof}

\subsection{Some linear algebra} \label{ss:linalg}

In this section, we work over an algebraically closed field $\bk$ of arbitrary characteristic.
Consider linear maps $f \colon V_0 \to V_1$ and $g \colon V_1 \to V_0$ such that $V_0$ and $V_1$ are finite-dimensional vector spaces. By picking bases, what form can we put the matrices in?

To answer this question, note that the tuples $(V_0, V_1, f, g)$ form an abelian category $\cA$. Concisely, we can regard $V_0 \oplus V_1$ as a $\bZ/2$-graded module over $\bk[t]$, where $t$ has degree~1; $t$ acts on $V_0$ by $f$ and on $V_1$ by $g$. We define some basic objects in this category.
\begin{itemize}
\item For $\lambda \in \bk$ let $A_n(\lambda)$ be the object with $V_0=V_1=\bk^n$, and where $f$ is the identity matrix and $g$ is a single Jordan block with eigenvalue $\lambda$.
\item Let $A_n(\infty)$ be the object with $V_0=V_1=\bk^n$, and where $f$ is a single nilpotent Jordan block and $g$ is the identity.
\item Let $B_n$ be the module $\bk[t]/(t^{2n+1})$. This has basis $x_0, \ldots, x_{2n}$, where $x_i$ has parity $i$, and $tx_i=x_{i+1}$ for $i<2n$, and $tx_{2n}=0$.
\end{itemize}
For an object $M$ of $\cA$, we let $M[1]$ be the object where the even and odd pieces of $M$ are swapped. One easily verifies that $A_n(\lambda)[1]$ is isomorphic to $A_n(\lambda^{-1})$ for $\lambda \in \bk \cup \{\infty\}$.

\begin{proposition} \label{prop:indecomp}
The $A_n(\lambda)$, $B_n$, and $B_n[1]$ are exactly the indecomposable objects of $\cA$.
\end{proposition}

\begin{proof}
Consider an object $(V_0, V_1, f, g)$. Then $V_0$ and $V_1$ are $\bk[t^2]$-modules, and $f$ and $g$ are maps of modules such that $fg=gf=t^2$. By the structure theorem for modules over a PID, $V_0$ decomposes as $V_0' \oplus V_0''$ where $t^2$ is invertible on $V_0'$ and nilpotent on $V_0''$; of course, $V_1$ decomposes similarly. Since $f$ and $g$ are maps of modules, they respect these decompositions. Thus the whole object decomposes into two pieces, and it suffices to consider these pieces separately.

First suppose that $t^2$ is invertible. Since $fg=t^2$, it follows that $f \colon V_0 \to V_1$ is an isomorphism. We thus see that our object is isomorphic to the object $(V_1, V_1, \id, fg)$, which is easily seen to decompose in terms of the $A_n(\lambda)$ with $\lambda \in \bk \setminus \{0\}$.

Now suppose that $t^2$ is nilpotent. Then the usual proof of the structure theorem for modules over a PID applies, and we see that our module decomposes into cyclic modules, which are $B_n$, $B_n[1]$, $A_n(0)$, or $A_n(\infty)$.
\end{proof}

\begin{remark}
The category $\cA$ is the representation category of the extended Dynkin quiver $\tilde{A}_1$. For more in this direction, see \cite[\S 7.1]{DerksenWeyman}.
\end{remark}

\section{Grothendieck--Springer theory for the super Grassmannian} \label{s:gs}

In this section we study the Grothendieck--Springer theory associated to the super Grassmannian $\Gr_{r|s}(\bC^{n|m})$ (see \S \ref{ss:gs}). We prove a number of results about the various spaces. These results are used to compute the coherent cohomology of $\Gr_{r|s}(\bC^{n|m})$ in the next section. This section takes place entirely in the world of ordinary (non-super) mathematics.

\subsection{Statement of results} \label{ss:detvar}

We introduce a number of objects:
\begin{itemize}
\item Let $V_0$ and $V_1$ be complex vector spaces of dimensions $n$ and $m$.
\item Let $r$ and $s$ be non-negative integers with $r \ge s$. Put
\begin{displaymath}
\delta = \begin{cases}
m-n+r-s & \text{if $r-s>n-m$} \\
0 & \text{if $n-m \ge r-s$} \end{cases}.
\end{displaymath}
Note that $m-\delta \le \min(n,m)$.
\item Put $W_0=\Hom(V_0,V_1)$ and $W_1=\Hom(V_1,V_0)$ and $W=W_0 \times W_1$. We regard these as affine varieties.
\item Let $Z_0 \subset W_0$ be the determinantal variety consisting of linear maps of rank $\le m-\delta$. Let $Z_1=W_1$ and put $Z=Z_0 \times Z_1$. (Note: when $\delta=0$ we have $Z_0=W_0$. This case is still interesting, though somewhat simpler.)
\item Letting $f \colon V_0 \otimes \cO_Z \to V_1 \otimes \cO_Z$ and $g \colon V_1 \otimes \cO_Z \to V_0 \otimes \cO_Z$ denote the universal linear maps, let $\chi(u) \in \cO_Z[u]$ be the characteristic polynomial of $fg$, and let $\ol{\chi}(u)=u^{-\delta} \chi(u)$. Since $fg$ has rank $\le m-\delta$, it follows that $\ol{\chi}(u)$ is a polynomial.
  
\item Let $\tilde{Z}$ be the affine scheme whose coordinate ring is the $(s,m-\delta-s)$-factorization ring for $\ol{\chi}(u)$ introduced in \S \ref{ss:factor}. Thus for a $\bC$-algebra $T$, a $T$-point of $\tilde{Z}$ is a triple $(f,g,p)$ where $(f,g)$ is a $T$-point of $Z$ and $p=p(u)$ is a degree $s$ monic polynomial over $T$ that divides $\ol{\chi}(u) \in T[u]$. There is a natural map $\tilde{Z} \to Z$ given by forgetting $p$.
  
\item Let $Y$ be the scheme defined as follows: a $T$-point of $Y$ is a tuple $(f,g,R_0,R_1)$ where:
\begin{itemize}
\item $R_0 \subset (V_0)_T$ is a $T$-submodule that is locally a rank $r$ summand.
\item $R_1 \subset (V_1)_T$ is a $T$-submodule that is locally a rank $s$ summand.
\item $f \colon (V_0)_T \to (V_1)_T$ is a map of $T$-modules such that $f(R_0) \subseteq R_1$.
\item $g \colon (V_1)_T \to (V_0)_T$ is a map of $T$-modules such that $g(R_1) \subseteq R_0$.
\end{itemize}
One easily sees that $Y$ is the total space of a vector bundle over $\Gr_r(V_0) \times \Gr_s(V_1)$, and is thus smooth and irreducible.
\item Let $\pi \colon Y \to \tilde{Z}$ be the map taking $(f,g,R_0,R_1)$ to $(f,g,p)$, where $p$ is the characteristic polynomial of $fg$ on $R_1$. We show that this is well-defined in Corollary~\ref{cor:tildeZ-divide}.
\end{itemize}
This set-up turns out to be the Grothendieck--Springer theory associated to the super Grassmannian, as explained in \S \ref{ss:supergrass}. The purpose of this section is to study the above situation, especially the varieties $Y$ and $\tilde{Z}$. Our main results are summarized in the following theorem:

\begin{theorem} \label{thm:detvar}
We have the following:
\begin{enumerate}
\item $\tilde{Z}$ is integral and has rational singularities (and is thus normal and Cohen--Macaulay).
\item The map $\tilde{Z} \to Z$ is finite and flat; in fact, $\cO_{\tilde{Z}}$ is a free $\cO_Z$-module of rank $\binom{m-\delta}{s}$.
\item The graded ring $\cO_{\tilde{Z}} \otimes_{\cO_Z} \bC$ is isomorphic to $\rH^*_{\sing}(\Gr_s(\bC^{m-\delta}), \bC)$.
\item The map $\pi \colon Y \to \tilde{Z}$ is projective and satisfies $\pi_*(\cO_Y)=\cO_{\tilde{Z}}$ and $\rR^i \pi_*(\cO_Y)=0$ for $i>0$. If $\delta>0$ or $r=s$ then $\pi$ is birational. If $\delta=0$ then there is an open dense subset $U$ of $\tilde{Z}$ such that $\pi^{-1}(U)$ is isomorphic to $U \times \Gr_{r-s}(\bC^{n-m})$.
\end{enumerate}
\end{theorem}

\begin{remark} \label{rmk:proof}
The proofs for the various parts are spread out:

(a): Integrality is the content of \S\ref{ss:integrality} and rational singularities is proven in Proposition~\ref{prop:Z-rat-sing}.

(b) and (c) are deduced in \S\ref{ss:first}.

(d): The first claim is Proposition~\ref{prop:pushforward-Z}, the second claim is Proposition~\ref{prop:birational}, and the third claim is Proposition~\ref{prop:delta0-fiber}.
\end{remark}

\begin{remark}
We make a clarifying remark about gradings. As $W$ is a vector space, it carries a natural $\bG_m$ action via usual scalar multiplication, and this action induces all the gradings. In terms of rings, we have
\begin{displaymath}
\cO_{W_0}=\Sym(V_0 \otimes V_1^*), \quad \cO_{W_1}=\Sym(V_1 \otimes V_0^*), \quad \cO_W=\cO_{W_0} \otimes \cO_{W_1},
\end{displaymath}
and the elements of $V_0 \otimes V_1^*$ and $V_1 \otimes V_0^*$ are given degree~1. The space $Y$ is a subbundle of $X_0 \times W$, where $X_0=\Gr_r(V_0) \times \Gr_s(V_1)$; again, the $\bG_m$ action on $W$ induces the grading on $\cO_{\tilde{Z}}=\Gamma(Y, \cO_Y)$. Precisely, writing $Y=\Sym(\eta)$, where $\eta$ is a quotient of $\cO_{X_0} \otimes W^*$, the degree $d$ piece of $\cO_{\tilde{Z}}$ is $\Gamma(X_0, \Sym^d(\eta))$. We note that the gradings we use do not necessarily coincide with the gradings induced by the central tori in $\GL(V_0)$ and $\GL(V_1)$.
\end{remark}

\subsection{Some auxiliary spaces}

To prove Theorem~\ref{thm:detvar}, we will use a few auxiliary spaces, which we now introduce.  Conceptually, these spaces come from the Grothendieck--Springer theory associated to a flag supervariety.
\begin{itemize}
\item Let $\tilde{\cZ}$ be the affine scheme whose coordinate ring is the splitting ring for $\ol{\chi}(u)$ introduced in \S \ref{ss:split}.
  \item Let $\cY$ be the scheme defined as follows: a $T$-point of $\cY$ is a tuple $(f,g,F_\bullet,G_\bullet)$ where:
    \begin{itemize}
    \item $F_{r-s} \subset F_{r-s+1} \subset \cdots \subset F_{r-s+m-\delta} \subset (V_0)_T$ is a flag of $T$-submodules which are locally summands with ranks prescribed by the subscript.
    \item $0=G_0 \subset G_1 \subset \cdots \subset G_{m-\delta} \subset (V_1)_T$ is a flag of $T$-submodules which are locally summands with ranks prescribed by the subscript.
\item $f \colon (V_0)_T \to (V_1)_T$ is a map of $T$-modules such that $f(F_{i+r-s}) \subseteq G_i$ for all $i \ge 0$.
\item $g \colon (V_1)_T \to (V_0)_T$ is a map of $T$-modules such that $g(G_i) \subseteq F_{i+r-s}$ for all $i \ge 0$.
\end{itemize}
In fact, $\cY$ is the total space of a vector bundle over a product of partial flag varieties $\Fl(r-s,\dots,r-s+m-\delta;V_0) \times \Fl(1,2,\dots,m-\delta;V_1)$.
\item Let $\rho \colon \cY \to \tilde{\cZ}$ be the map taking $(f,g,F_\bullet,G_\bullet)$ to $(f,g,\prod_{i=1}^{m-\delta} (u-\lambda_i))$, where $\lambda_i$ is the eigenvalue of $fg$ on $G_i/G_{i-1}$. We show that this is well-defined in Proposition~\ref{prop:tildecZ-divide}.
\end{itemize}
We prove the following analog of Theorem~\ref{thm:detvar}:

\begin{theorem} \label{thm:detvar2}
We have the following:
\begin{enumerate}
\item $\tilde{\cZ}$ is integral and has rational singularities (and is thus normal and Cohen--Macaulay).
\item The map $\tilde{\cZ} \to Z$ is finite and flat; in fact, $\cO_{\tilde{\cZ}}$ is a free $\cO_Z$-module of rank $(m-\delta)!$.
\item The graded ring $\cO_{\tilde{\cZ}} \otimes_{\cO_Z} \bC$ is isomorphic to $\rH^*_{\sing}(\Fl(\bC^{m-\delta}), \bC)$.
\item The map $\rho \colon \cY \to \tilde{\cZ}$ is projective and satisfies $\rho_*(\cO_{\cY})= \cO_{\tilde{\cZ}}$ and $\rR^i \rho_*(\cO_{\cY}) =0$ for $i>0$. If $\delta>0$ or $r=s$ then $\rho$ is birational. If $\delta=0$ then there is an open dense subset $U$ of $\tilde{\cZ}$ such that $\rho^{-1}(U)$ is isomorphic to $U \times \Gr_{r-s}(\bC^{n-m})$. 
\end{enumerate}
\end{theorem}

See Remark~\ref{rmk:proof} for the locations of the proofs of the various statements (we handle this and the previous case in parallel).

\subsection{First observations} \label{ss:first}

It is well-known that $Z_0$ has rational singularities \cite[Proposition 6.2.3]{weyman}, and is in particular normal and Cohen--Macaulay. Since $Z_1$ is an affine space, it follows that the same holds for $Z=Z_0 \times Z_1$. Theorem~\ref{thm:detvar}(b,c) and Theorem~\ref{thm:detvar2}(b,c), as well as the fact that $\tilde{Z}$ and $\tilde{\cZ}$ are Cohen--Macaulay, follow immediately from Propositions~\ref{prop:split}, \ref{prop:split-flag}, \ref{prop:fact}, and~\ref{prop:fact-grass}. The following proposition and corollary ensure that $\rho$ and $\pi$ are well-defined.

\begin{proposition} \label{prop:tildecZ-divide}
Suppose $(f,g,F_\bullet,G_\bullet)$ is a $T$-point of $\cY$. Let $\lambda_i$ be the eigenvalue of $fg$ on $G_i/G_{i-1}$. Then $\prod_{i=1}^{m-\delta}(u-\lambda_i) = \ol{\chi}(u)$.
\end{proposition}

\begin{proof}
If $\delta=0$, then $G_{m-\delta}=V_1$ and $\ol{\chi}(u)=\chi(u)$, so there is nothing to prove. Otherwise, suppose that $\delta = m-n+r-s > 0$. Then $F_{r-s+m-\delta}=V_0$, so that the image of $f$ is contained in $G_{m-\delta}$. So $fg$ induces the zero endomorphism on $V_1/G_{m-\delta}$, which implies that the characteristic polynomial of $fg$ on $G_{m-\delta}$ is $\ol{\chi}(u)$, and which proves the result.
\end{proof}

\begin{corollary} \label{cor:tildeZ-divide}
Suppose $(f,g,R_0,R_1)$ is a $T$-point of $Y$. Let $p$ be the characteristic polynomial of $fg$ on $R_1$. Then $p(u)$ divides $\ol{\chi}(u)$.
\end{corollary}

\begin{proof}
This is immediate from the fact that we have a surjective map $\cY \to Y$ given by $(f,g,F_\bullet,G_\bullet) \mapsto (f,g,F_r,G_s)$.
\end{proof}

\subsection{Integrality} \label{ss:integrality}

We now show that $\tilde{Z}$ and $\tilde{\cZ}$ are integral. Let $\Delta \in \cO_Z$ be the discriminant of $\ol{\chi}$. We begin with the following:

\begin{proposition} \label{prop:Deltane0}
We have $\Delta\ne 0$, and so $\tilde{\cZ}$ and $\tilde{Z}$ are reduced.
\end{proposition}

\begin{proof}
We construct a $\bC$-point of $Z$ where $\Delta \ne 0$. Let $x_1, \ldots, x_n$ and $y_1, \ldots, y_m$ be bases for $V_0$ and $V_1$. Let $\lambda_1, \ldots, \lambda_{m-\delta}$ be distinct complex numbers. Define $f \colon V_0 \to V_1$ by $f(x_i)=\lambda_i y_i$ for $1 \le i \le m-\delta$ and $f(x_i)=0$ for $i>m-\delta$. Define $g \colon V_1 \to V_0$ by $g(x_i)=y_i$ for $1 \le i \le m-\delta$ and $g(y_i)=0$ for $i>m-\delta$. Note that $(f,g)$ defines a $\bC$-point of $Z$. The composition $fg$ has characteristic polynomial $\chi(u)=u^{\delta} \prod_{i=1}^{m-\delta} (u-\lambda_i)$. Thus $\Delta \ne 0$ at $(f,g)$. Since $Z$ is integral, it follows that $\Delta$ is a non-zerodivisor, and so $\tilde{\cZ}$ is reduced (Proposition~\ref{prop:split}(e)), and so $\tilde{Z}$ is as well by Proposition~\ref{prop:fact}(e).
\end{proof}

To complete the proof of integrality, it suffices to show that $\tilde{Z}$ and $\tilde{\cZ}$ are irreducible. Since $Y$ and $\cY$ are irreducible, this follows from the following result:

\begin{proposition} \label{prop:rho-surj}
The maps $\rho \colon \cY \to \tilde{\cZ}$ and $\pi \colon Y \to \tilde{Z}$ are surjective.
\end{proposition}

\begin{proof}
Let $(f,g,\prod_{i=1}^{m-\delta}(u-\lambda_i))$ be a $\bC$-point of $\tilde{\cZ}$. Recall the category $\cA$ from \S \ref{ss:linalg}. Let $M$ be the object $(V_0,V_1,f,g)$. We inductively construct a chain of subobjects $M^0 \subset \cdots \subset M^{m-\delta}$ of $M$, with the following properties:
\begin{enumerate}
\item $\dim(M^i_0)=i+r-s$ and $\dim(M^i_1)=i$.
\item For $1 \le i \le m-\delta$, the eigenvalues of $t^2$ on $M^i/M^{i-1}$ are $\lambda_i$.
\end{enumerate}
To begin, we take $M^0_0$ to be a subspace of $\ker(f)$ of dimension $r-s$ (which exists since $f \in Z_0$) and $M^0_1=0$. Suppose now that we have defined $M^{i-1}$ for some $1 \le i \le m-\delta$.

We claim that there is an $\cA$-subobject of $M/M^{i-1}$ of dimension $1 \vert 1$ such that the eigenvalues of $t^2$ are both $\lambda_i$; granted this, we take $M^i/M^{i-1}$ to be this subobject. By considering the characteristic polynomial of $t^2=fg$ on $M_1$, we see that $\lambda_i$ occurs as an eigenvalue of $t^2$ on $M^i_1/M^{i-1}_1$. If $\lambda_i \ne 0$ then the claim follows from the form of indecomposables given in \S \ref{ss:linalg}. Suppose now $\lambda_i=0$. For an object $N$ of $\cA$, let $\Phi(N)$ be the generalized $0$-eigenspace of $t^2$ and let $\Psi(N)$ be the sum of the generalized $\lambda$-eigenspaces with $\lambda \ne 0$. We have $\dim(\Psi(N)_0)=\dim(\Psi(N)_1)$. Let $k$ be the number of $\lambda_i$'s that are zero. Then
\begin{displaymath}
\dim(\Phi(M)_1)=\delta+k, \quad \dim(\Psi(M)_1)=m-(\delta+k), \quad \dim(\Phi(M)_0)=n-m+\delta+k.
\end{displaymath}
We thus find
\begin{displaymath}
\dim(\Phi(M/M^0)_0)=n-m+s-r+\delta+k \ge k.
\end{displaymath}
Now, since $\lambda_i=0$, it follows that $<k$ of $\lambda_1, \ldots, \lambda_{i-1}$ are~0, and so $\dim(\Phi(M/M^{i-1})_0)>0$. We thus see that $\Phi(M/M^{i-1})$ has non-zero even and odd parts, and thus (by the classification of indecomposables) has a subobject of dimension $1|1$.

We now define a $\bC$-point $(f,g,F_{\bullet},G_{\bullet})$ of $\cY$ by taking $F_{r-s+i}=M^i_0$ and $G_i=M^i_1$. It is clear that this is indeed a $\bC$-point of $\cY$ and lifts the point of $\tilde{\cZ}$ under consideration. Thus $\rho$ is surjective. We have a commutative diagram
\begin{displaymath}
\xymatrix{
\cY \ar[r] \ar[d]_{\rho} & Y \ar[d]^{\pi} \\
\tilde{\cZ} \ar[r] & \tilde{Z} }
\end{displaymath}
where the top map takes $(f,g,F_{\bullet},G_{\bullet})$ to $(f,g,F_r,G_s)$ and the bottom map is the natural one. Since the bottom map is clearly surjective, so is $\pi$.
\end{proof}

\subsection{Normality}

We now aim to show that $\tilde{Z}$ and $\tilde{\cZ}$ are normal. We will use the criterion in  Proposition~\ref{prop:norm-crit2}. We therefore begin by studying the geometry of $V(\Delta)$ and $V(\Delta,\partial \Delta)$.

\begin{proposition}
Let $(f,g)$ be a $\bC$-point of $Z$. We can choose bases of $V_0$ and $V_1$ such that the matrices for $f$ and $g$ have the form
\begin{displaymath}
\begin{pmatrix}
0 & \ast \\
0 & A
\end{pmatrix}
\qquad
\begin{pmatrix}
\ast & \ast \\
0 & B
\end{pmatrix}
\end{displaymath}
where $A$ and $B$ are upper-triangular square matrices with $m-\delta$ rows and columns (this determines the sizes of the other blocks). Moreover, if $(f,g) \in V(\Delta)$ then one can also assume that the first two diagonal entries of $AB$ are equal.
\end{proposition}

\begin{proof}
Since $f$ has rank $\le m-\delta$ its kernel has dimension $\ge n-m+\delta$; similarly, since $fg$ has rank $\le m-\delta$, its kernel has dimension $\ge \delta$. Let $U_1 \subset \ker(fg)$ be a subspace of dimension $\delta$. Then $g(U_1)$ is a subspace of $\ker(f)$ of dimension $\le \delta$. Let $U_0 \subset \ker(f)$ be a subspace of dimension $n-m+\delta$ containing $g(U_1)$. Now, pick a basis $x_1, \ldots, x_n$ of $V_0$ such that $x_1, \ldots, x_{n-m+\delta}$ is a basis of $U_0$, and pick a basis $y_1, \ldots, y_m$ of $V_1$ such that $y_1, \ldots, y_{\delta}$ is a basis for $U_1$. In these bases, the matrices for $f$ and $g$ have the stated form, except that $A$ and $B$ may not be upper-triangular. Let $U_0'$ be the span of $x_{n-m+\delta+1}, \ldots, x_n$, and let $U_1'$ be the span of $y_{\delta+1}, \ldots, y_m$. We can then regard $(A, B, U_0', U_1')$ as an object of the category considered in \S \ref{ss:linalg}. By decomposing into indecomposables, and choosing bases for these indecomposables, we can make $A$ and $B$ upper-triangular. (The one subtlety here is that the matrices for $B_n$ and $B_n'$ are not square. However, since $\dim(U_0')=\dim(U_1')$ whenever we have a $B_n$ we can pair it with some $B_m'$, and the matrices for $B_n \oplus B_m'$ are square and can be taken to be upper-triangular.)

Finally, suppose that $(f,g) \in V(\Delta)$. If $\ol{\chi}$ has a non-zero repeated root $\lambda$ then by taking the $A_n(\lambda)$'s first in the decomposition, we see that the first two diagonal entries of $AB$ are $\lambda$. Otherwise, 0 is a repeated root of $\ol{\chi}$, and by putting the $A_n(0)$, $A_n(\infty)$, $B_n$, and $B_n$'s first in the decomposition, the first two diagonal entries of $AB$ are~0.
\end{proof}

\begin{proposition}
The closed set $V(\Delta) \subset Z$ is irreducible; in fact, any two points of $V(\Delta)$ can be joined by an irreducible rational curve in $V(\Delta)$.
\end{proposition}

\begin{proof}
Fix bases for $V_0$ and $V_1$ and let $(f_i,g_i)$ for $i=1,2$ be two points of $V(\Delta)$. Applying the previous proposition, we can find $\gamma_i \in \GL(V_0) \times \GL(V_1)$ such that
\begin{displaymath}
\gamma_i f_i = \begin{pmatrix} 0 & C_i \\ 0 & A_i \end{pmatrix}, \qquad
\gamma_i g_i = \begin{pmatrix} D_i & E_i \\ 0 & B_i \end{pmatrix}
\end{displaymath}
where $A_i$ and $B_i$ are upper-triangular square matrices of size $m-\delta$, and the first two diagonal entries of $A_iB_i$ are equal. We define matrices $A(t), \ldots, E(t)$ with polynomial entries such that:
\begin{itemize}
\item At $t=0$ they coincide with $A_1, \ldots, E_1$
\item At $t=1$ they coincide with $A_2, \ldots, E_2$.
\item The matrices $A(t)$ and $B(t)$ are upper-triangular.
\item The first two diagonal entries of $A(t) B(t)$ agree.
\end{itemize}
We take $C(t)=tC_1+(1-t)C_2$. We use the same formula everywhere else, except for the first two diagonal entries of $A(t)$ and $B(t)$. Consider the subvariety of $\bA^4$ defined by $\alpha_1 \alpha_2=\beta_1 \beta_2$. Any two points on this variety can be joined by a map from $\Spec(\bC[t])$, as it is the affine cone over $\bP^1 \times \bP^1$. We take $(A_{1,1}(t), A_{2,2}(t), B_{1,1}(t), B_{2,2}(t))$ to be such a curve joining $(A_{i,1,1}, A_{i,2,2}, B_{i,1,1}, B_{i,2,2})$.

Now, define
\begin{displaymath}
F(t) = \begin{pmatrix} 0 & C(t) \\ 0 & A(t) \end{pmatrix}, \qquad
G(t) = \begin{pmatrix} D(t) & E(t) \\ 0 & B(t) \end{pmatrix}.
\end{displaymath}
Since $\ker(F(t))$ contains the first $n-m+\delta$ basis vectors, the rank of $F(t)$ is $\le m-\delta$. Thus $(F(t), G(t))$ defines a point of $Z$. The polynomial $\ol{\chi}_{F(t), G(t)}$ is the characteristic polynomial of $A(t) B(t)$, which has a repeated root since the first two diagonal entries coincide. Thus $(F(t), G(t))$ is a point of $V(\Delta)$.

Finally, let $\gamma \colon \Spec(\bC[t]) \to \GL(V_0) \times \GL(V_1)$ be a curve such that $\gamma(0)=\gamma_1$ and $\gamma(1)=\gamma_2$. Then $\gamma (F(t), G(t))$ is a curve in $V(\Delta)$ joining $(f_1,g_1)$ to $(f_2,g_2)$. This completes the proof.
\end{proof}

\begin{proposition}
$\tilde{\cZ}$ and $\tilde{Z}$ are normal.
\end{proposition}

\begin{proof}  
It suffices to show that $\tilde{\cZ}$ is normal (due to Proposition~\ref{prop:fact}(e)). If $m-\delta \le 1$ then $\deg(\ol{\chi}) \le 1$ and $\tilde{\cZ}=Z$ is normal. We thus assume that $m-\delta \ge 2$ in what follows.

We verify the conditions of Proposition~\ref{prop:norm-crit2}. We already know that $\cO_Z$ is normal and that $\Delta$ is a non-zerodivisor (Proposition~\ref{prop:Deltane0}). It thus suffices to show that $V(\Delta,\partial \Delta)$ has codimension $\ge 2$ in $Z$. Since $Z$ is normal, its singular locus has codimension $\ge 2$. It is thus enough to show that $Z_{\rm reg} \cap V(\Delta, \partial \Delta)$ has codimension $\ge 2$ in $Z_{\rm reg}$, where $Z_{\rm reg}$ is the regular locus (we will use below that $Z_{\rm reg}$ is the set of matrices with maximal possible rank). The set $Z_{\rm reg} \cap V(\Delta, \partial \Delta)$ is closed in $Z_{\rm reg}$ (see the first paragraph of \S \ref{ss:norm2}) and contained in $V(\Delta)$. Since $V(\Delta)$ is irreducible of codimension~1, it thus suffices to show that $Z_{\rm reg} \cap V(\Delta, \partial \Delta)$ is a proper subset of $Z_{\rm reg} \cap V(\Delta)$. We do this by writing down a point in $Z_{\rm reg} \cap V(\Delta)$ that does not belong to $V(\Delta, \partial \Delta)$.

Pick bases for $V_0$ and $V_1$. We define a $\bC[\epsilon]/(\epsilon^2)$ point of $Z$ by
\begin{displaymath}
f=\begin{pmatrix} 0 & 0 \\ 0 & A \end{pmatrix} \qquad
g=\begin{pmatrix} 0 & 0 \\ 0 & I \end{pmatrix}
\end{displaymath}
where $A$ and $I$ are square matrices of size $m-\delta$, $I$ is the identity matrix, and
\begin{displaymath}
A = \begin{pmatrix}
\lambda & 1 \\
\epsilon & \lambda \\
&& D \end{pmatrix}
\end{displaymath}
where $\lambda \in \bC \setminus \{0\}$ and $D$ is a diagonal matrix with distinct non-zero complex entries that are not $\lambda$. Let $x=(\ol{f}, \ol{g})$ be the $\bC$-point of $Z$ obtained by putting $\epsilon=0$. It is clear that $x$ belongs to $V(\Delta)$, since $\ol{f} \ol{g}$ has a repeated eigenvalue. Since $\ol{f}$ has maximal possible rank, it is a smooth point of $Z_0$, and so $x \in Z_{\rm reg}$. The value of $\Delta$ on $(f,g)$ is a non-zero multiple of $\epsilon$: the key point is that the discriminant of the characteristic polynomial of the top $2 \times 2$ block of $A$ is $4 \epsilon$. This shows that $d\Delta$ is non-zero at $(f,g)$, and so $x \not\in V(\Delta, \partial \Delta)$.
\end{proof}

\subsection{Rational singularities}

We now show that $\tilde{Z}$ and $\tilde{\cZ}$ have rational singularities. To this end, we introduce the following closed subsets of $Z$:
\begin{itemize}
\item Let $D_1$ be the locus of points such that $0$ is a root of $\ol{\chi}(u)$. 
\item Let $D_2$ be the locus such that $\ol{\chi}(u)$ has a repeated root. 
\item Let $D_3 \subset D_2$ be the locus where there is
  \begin{itemize}
  \item a triple root, or 
  \item two repeated roots, or 
  \item a unique repeated root, but the corresponding Jordan block of $fg$ is a scalar.
  \end{itemize}
This set is closed, as we now explain. The first two conditions define a closed set. Letting $U$ be the complement, it suffices to show that $D_3 \cap U$ is closed in $U$. Let $V \subset U \times \bA^1$ be the set of pairs $(f,g,\lambda)$ where $\lambda$ is a root of $\ol{\chi}$ and its derivative. The map $V \to U$ is finite and injective with image $V(\Delta) \cap U$. Let $V_1 \subset V$ be the closed set where $fg-\lambda$ has nullity $\ge 2$. Then $D_3 \cap U$ is the image of $V_1$ in $U$, and thus closed.
\item Let $D_4=D_1 \cap D_2$ and let $D_5=D_1 \cup D_3$.
\end{itemize}
Put $U_i=Z \setminus D_i$. For any space $X$ over $Z$, we let $D_i(X)$ or $U_i(X)$ denote the inverse image of $D_i$ or $U_i$ in $X$.

\begin{proposition} \label{prop:gl-codim2}
The sets $D_3(\cY)$ and $D_4(\cY)$ have codimension $\ge 2$ in $\cY$.
\end{proposition}

\begin{proof}
  By equivariance, the restrictions of $D_3(\cY)$ or $D_4(\cY)$ to any fiber over $\Fl(r-s,\dots,r-s+m-\delta; V_0) \times \Fl(1,\dots,m-\delta;V_1)$ is isomorphic to any other. So it suffices to show that in a given fiber of the vector bundle, these restrictions have codimension $\ge 2$. In that case, by picking bases of $V_0$ and $V_1$ adapted to the particular pair of flags, $f$ and $g$ have the following form:
  \begin{itemize}
  \item the first $r-s$ columns of $f$ are $0$,
  \item the next $m-\delta$ columns of $f$ are upper-triangular, 
  \item the remaining columns of $f$ (if they exist) are arbitrary,
  \item the first $r-s$ rows of $g$ are arbitrary,
  \item the next $m-\delta$ rows of $g$ are upper-triangular,
  \item the remaining rows of $g$ (if they exist) are 0.
  \end{itemize}
  So $fg$ is determined by the contents of the upper-triangular $m-\delta$ columns of $f$ and the upper-triangular $m-\delta$ rows of $g$. Let $x_1,\dots,x_{m-\delta}$ and $y_1,\dots,y_{m-\delta}$ be the corresponding diagonal entries of $f$ and $g$, respectively.
  
  Hence over this fiber, the condition that a root is repeated has one component for each equation $x_iy_i = x_jy_j$ for $i<j$. Then $D_4(\cY)$ is cut out in this component by $\prod_i x_iy_i$, and hence has codimension 2. Having a triple root or two repeated roots corresponds to imposing two equations $x_iy_i=x_jy_j$ and $x_ky_k=x_\ell y_\ell$, so has codimension 2. If there is a repeated root which is unique, the condition that the Jordan block is scalar is a nontrivial condition on an irreducible component and hence also has codimension $\ge 2$.
\end{proof}

\begin{proposition} \label{prop:birational}
Assume that $r=s$ or $\delta>0$. Then $\rho \colon U_5(\cY) \to U_5(\tilde{\cZ})$ is an isomorphism, and so $\rho$ is birational. Similarly, $\pi \colon U_5(Y) \to U_5(\tilde{Z})$ is an isomorphism, and so $\pi$ is birational.
\end{proposition}

\begin{proof} 
We first prove the statement about $\rho$. Since $\tilde{\cZ}$ is normal and $\rho$ is projective, it suffices to show that $\rho$ induces an isomorphism on $\bC$-points on these two sets. By Proposition~\ref{prop:rho-surj}, we know that $\rho$ is surjective, so it suffices to prove injectivity.

First suppose that $r=s$. Consider a point $(f,g,\prod_{i=1}^{m-\delta}(u-\lambda_i))$ in $U_5(\tilde{\cZ})$, and let $(f,g,F_{\bullet}, G_{\bullet})$ be an inverse image under $\rho$ . If the $\lambda_i$ are distinct then $G_i$ must be the span of the $fg$-eigenspaces for $\lambda_1,\dots,\lambda_i$ and $F_i$ must be the span of the $gf$-eigenspaces for $\lambda_1,\dots, \lambda_i$. Now suppose there is a repeated root. Since our point belongs to $U_5(\tilde{\cZ})$, there is only one repeated root, and the corresponding Jordan block for $fg$ is non-scalar. The spaces $F_i$ and $G_i$ are again uniquely determined: for the repeated eigenvalue, we use the eigenvector first, and the generalize eigenvector second. This shows that there is a unique inverse image under $\rho$.

Now suppose instead that $\delta>0$. As before, there is a unique choice of the flag $G_i$. Also by our definition of $U_1$, we have $\rank(fg)=m-\delta=\rank(gf)$ and hence $\dim \ker(gf) = r-s$. Hence $F_i$ must be the span of $\ker(gf)$ and the eigenspaces for $\lambda_1,\dots,\lambda_i$.

The claim about $\pi$ follows from similar reasoning.
\end{proof}

\begin{corollary} \label{cor:gl-ratsing}
  Assume that $r=s$ or $\delta>0$. Then $U_2(\tilde{\cZ})$ and $U_5(\tilde{\cZ})$ have rational singularities.
\end{corollary}
  
\begin{proof}
By Proposition~\ref{prop:birational}, $U_5(\tilde{\cZ})$ is smooth, and hence has rational singularities. Next, $\tilde{\cZ} \to Z$ is \'etale over $U_2$, and since $Z$ has rational singularities, the same is true for $U_2(\tilde{\cZ})$.
\end{proof}

\begin{proposition} \label{prop:Z-rat-sing}
The varieties $\tilde{\cZ}$ and $\tilde{Z}$ have rational singularities.
\end{proposition}

\begin{proof}
First suppose that $\delta>0$ or $r=s$. By Corollary~\ref{cor:gl-ratsing}, $\tilde{\cZ}$ has rational singularities on $U_2(\tilde{\cZ}) \cup U_5(\tilde{\cZ})$. The complement is $D_3(\tilde{\cZ}) \cap D_4(\tilde{\cZ})$, and the inverse image of this set under $\rho$ has codimension $\ge 2$ by Proposition~\ref{prop:gl-codim2}. Hence $\tilde{\cZ}$ has rational singularities by Proposition~\ref{prop:rat-sing}. Next if $\delta=0$ and $r \ne s$, the variety $\tilde{\cZ}$ is the same as the variety $\tilde{\cZ}$ with $(r,s)$ changed to $(s,s)$, and thus has rational singularities as well. Finally, $\tilde{Z}$ is a quotient of $\tilde{\cZ}$ by a finite group, so has rational singularities by \cite[Proposition 5.13]{kollarmori}.
\end{proof}

\subsection{Cohomology of $Y$ and $\cY$}

We now finish off the proofs by computing the cohomology of $\cO_Y$ and $\cO_{\cY}$. The following result is required to treat the $\delta=0$ case, and is part of Theorem~\ref{thm:detvar}(d) and~\ref{thm:detvar2}(d). Recall that $\delta=0$ is equivalent to the condition $n-m \ge r-s$.

\begin{proposition} \label{prop:delta0-fiber} 
Suppose $\delta=0$.
\begin{enumerate}
\item Let $U \subset Z$ be the locus where $fg \colon V_1 \to V_1$ is an isomorphism, $\Delta \ne 0$, and $\ol{\chi}(0) \ne 0$. Then $U$ is an open dense subset of $Z$.
\item There is a vector bundle $\cE$ on $U$ of rank $n-m$ whose fiber at $(f,g)$ is the kernel of $f$.
\end{enumerate}
Let $U'$ be a dense open subset of $U$ where $\cE$ is trivial and let $\tilde{U}$ and $\tilde{\cU}$ be the inverse images of $U'$ in $\tilde{Z}$ and $\tilde{\cZ}$. Then
\begin{enumerate}[resume]
\item $\pi^{-1}(\tilde{U}) \cong \Gr_{r-s}(\bC^{n-m}) \times \tilde{U}$, and 
\item $\rho^{-1}(\tilde{\cU})\cong \Gr_{r-s}(\bC^{n-m}) \times \tilde{\cU}$.
\end{enumerate}
\end{proposition}

\begin{proof}
(a) It is clear that $U$ is open. We show that it is non-empty. Since $\delta=0$, we have $n \ge m$. Let $x_1, \ldots, x_n$ be a basis of $V_0$ and $y_1, \ldots, y_m$ a basis of $V_1$. Define $f \colon V_0 \to V_1$ by $f(x_i)=y_i$ for $1 \le i \le m$ and $f(x_i)=0$ for $i>m$. Define $g \colon V_1 \to V_0$ by $g(y_i)=\lambda_i x_i$ where $\lambda_1, \ldots, \lambda_m$ are distinct non-zero complex numbers. Then $(f,g)$ belongs to $U$.

(b) If $(f,g) \in U$ then $f$ is surjective and $\ker(f)$ has dimension $n-m$, so we can take $\cE$ to be the kernel of the generic map $f \colon V_0 \otimes \cO_U \to V_1 \otimes \cO_U$.

(c) Suppose $(f,g,p)$ is a point in $\tilde{U}$. Consider a point $(f,g,R_0,R_1)$ in $Y$ above $(f,g,p)$. Then $R_1$ is uniquely determined: it is the sum of the eigenspaces of $fg$ corresponding to eigenvalues that are roots of $p$. Since $fg$ is an isomorphism, $g$ is injective and its image is linearly disjoint from $\ker(f)$. Thus $R_0= g(R_1) \oplus K$ where $K$ is a subspace of $\ker(f)$ of dimension $r-s$. Moreover, any choice of $K$ leads to a point above $(f,g,p)$. We thus see that the fiber above $(f,g,p)$ is $\Gr_{r-s}(\bC^{n-m})$. The isomorphism $\pi^{-1}(\tilde{U})=\Gr_{r-s}(\bC^{n-m}) \times \tilde{U}$ comes from combining this analysis with a choice of trivialization of $\cE$ over $U'$.

(d) This is similar to (c): the flag $G_{\bullet}$ is uniquely determined, $F_{r-s}$ is an $r-s$ dimensional subspace of $\ker(f)$, and $F_{r-s+i}= g(G_i) + F_{r-s}$ for $1 \le i \le m-\delta$.
\end{proof}

\begin{proposition} \label{prop:pushforward-Z}
We have $\rR \pi_*(\cO_Y)=\cO_{\tilde{Z}}$ and $\rR \rho_*(\cO_{\cY})=\cO_{\tilde{\cZ}}$.
\end{proposition}

\begin{proof}
If $\delta>0$ or $r=s$ then $\pi$ and $\rho$ are birational (Proposition~\ref{prop:birational}), and the result follows from generalities on rational singularities (see \S \ref{ss:rat-sing}). For $\delta=0$, the result follows from Propositions~\ref{prop:rat-fiber} and~\ref{prop:delta0-fiber}.
\end{proof}

\section{Cohomology of the super Grassmannian} \label{s:coh}

\subsection{Statement of results}

We use notation as in \S \ref{ss:detvar}. Additionally, we introduce the following notation:
\begin{itemize}
\item Let $V$ be the super vector space $V_0 \oplus V_1$.
\item Let $X$ be the super Grassmannian $\Gr_{r|s}(V)$ (see \S \ref{ss:supergrass} for background).
\item Let $\bG$ be the super group $\GL(V)$, and let $\bG_0=\GL(V_0) \times \GL(V_1)$ be the underlying ordinary group.
\item Let $A=\rH^*_{\rm sing}(\Gr_s(\bC^{m-\delta}), \bC)$, regarded as a graded $\bC$-algebra.
\item Let $S=\Sym(W^*)$ be the coordinate ring of $W$.
\end{itemize}
The main result of this section is Theorem~\ref{mainthm}, which computes the cohomology of $\cO_X$. We restate the theorem here in our current notation:

\begin{theorem} \label{thm:grcoh}
  We have the following:
\begin{enumerate}
\item We have a natural isomorphism $\rH^*(X, \cO_X)^{\bG}=A$ of graded algebras.
\item There is a canonical graded $\bG$-subrepresentation $E$ of $\rH^*(X, \cO_X)$ such that the natural map $A \otimes E \to \rH^*(X, \cO_X)$ is an isomorphism.
\item We have a canonical isomorphism of $\bG_0$-representations
\begin{displaymath}
E^i = \bigoplus_{p \ge 0} \Tor_p^S(\cO_Z, \bC)_{i+p}.
\end{displaymath}
\end{enumerate}
\end{theorem}

We note that if $\delta=0$ then $Z=W$, and so $E^i=0$ for $i>0$ and $E^0=\bC$. Moreover, $\bG_0$ acts trivially on $E^0$, and so $\bG$ does as well. Thus, in this case, we have $\rH^*(X,\cO_X)=A$ with trivial $\bG$ action, as in Theorem~\ref{mainthm}(a).

A corollary of the theorem is that the $\bG_0$-action on the linear strands of the resolution of $\cO_Z$ extends to an action of $\bG$. This result was first proved by Pragacz--Weyman \cite{pragacz-weyman}. We discuss this in more detail in \S \ref{ss:pw}.

\subsection{Grothendieck--Springer theory} \label{ss:supergrass}

We now give some background on the super Grassmannian and connect it to varieties studied in \S \ref{s:gs}. While the super Grassmannian can be expressed in the form $\bG/\bP$ for an appropriate parabolic subgroup $\bP$ of $\bG$, we use an alternative approach here. We do this to be more direct, and also because we do not know of existing literature where quotients of supergroups have been treated carefully. We refer to \cite[Ch.~4, \S 3]{manin} for general background on the super Grassmannian.

The super Grassmannian $X=\Gr_{r|s}(V)$ is the super scheme representing the functor that attaches to a super algebra $T$ the set of $T$-submodules of $T \otimes V$ that are locally summands of rank $r|s$. It is not difficult to see that this functor is indeed representable. One can show that $X$ is smooth using the criterion for formal smoothness in the super setting; this is also proved in \cite{manin}. The super scheme $X$ is smooth of dimension $d_0|d_1$ where
\[
  d_0=r(n-r)+s(m-s), \qquad d_1=r(m-s)+ s(n-r).
\]
Let
\begin{displaymath}
0 \to \cR \to \cO_X \otimes V \to \cQ \to 0
\end{displaymath}
be the tautological sequence on $X$, so that $\cR$ is a vector bundle of rank $r|s$. Also, let $X_0$ be the ordinary scheme $\Gr_r(V_0)$ and let
\begin{displaymath}
0 \to \cR_0 \to \cO_{X_0} \otimes V_0 \to \cQ_0 \to 0
\end{displaymath}
be its tautological sequence. Similarly define $X_1 = \Gr_s(V_1)$ and let
\[
  0 \to \cR_1 \to \cO_{X_1} \otimes V_1 \to \cQ_1 \to 0
\]
be its tautological sequence. Restricting the functor of points of $X$ to ordinary algebras, one sees that $X_{\ord}=X_0 \times X_1$. One also finds $\cQ_{\ord}=\cQ_0 \oplus \cQ_1$, and similarly $\cR_{\ord}=\cR_0 \oplus \cR_1$. (We simply write $\cQ_0$ for the pullback of $\cQ_0$ to $X_0 \times X_1$, and similarly in other cases.)

We now determine $\gr(\cO_X)$. The result of the calculation below can be found in \cite{manin} but we prefer to give a short self-contained explanation. We begin with the following observation:

\begin{lemma} \label{lem:map-phi}
Let $T$ be a super algebra, let $J=J_T$, and let
\begin{displaymath}
0 \to M \to T \otimes V \to N \to 0
\end{displaymath}
be an exact sequence of $T$-modules. Then there exists a unique map of $T/J$-modules
\begin{displaymath}
\phi \colon M/JM \to JN/J^2N
\end{displaymath}
satisfying the following condition: if $x$ is a homogeneous element of $M$ of degree $d$ and $x=y_0+y_1$ with $y_i \in T_{d+i} \otimes V_i$ then $\phi(\ol{x})=\ol{y}_{d+1}$, where the bar denotes the image in the quotient module.
\end{lemma}

\begin{proof}
Define $\wt{\phi} \colon M \to J \otimes V$ as follows: given $x \in M_d$, write $x=y_0+y_1$ with $y_i \in T_{i+d} \otimes V_i$, and put $\wt{\phi}(x)=y_{d+1}$. Note that $y_{d+1} \in T_1 \otimes V_{d+1}$ does indeed belong to $J \otimes V$. One easily sees that $\wt{\phi}$ is a map of $T_0$-modules. Now, suppose $x=y_0+y_1$ is as above and $a \in T_1$. Then $ax=ay_0+ay_1$ has degree $d+1$ and so $\wt{\phi}(ax)=ay_d=ax-ay_{d+1}$ belongs to $JM+J^2 \otimes V$. It follows that $\wt{\phi}(JM) \subset JM+J^2 \otimes V$, and so $\wt{\phi}$ induces a $T_0$-linear map $\phi \colon M/JM \to JN/J^2N$. Since $T_0 \to T/J$ is surjective, it follows that $\phi$ is $T/J$-linear as well. It is clear that $\phi$ is the unique map satisfying the stated condition.
\end{proof}

\begin{proposition} \label{prop:grOX}
We have a natural isomorphism
\begin{displaymath}
\iota \colon (\cQ_1^* \otimes \cR_0) \oplus (\cQ_0^* \otimes \cR_1) \to \cJ/\cJ^2
\end{displaymath}
of coherent $\cO_{X_{\ord}}$-modules.
\end{proposition}

\begin{proof}
  Since $\cR/\cJ \cR \cong \cR_0 \oplus \cR_1$ and $\cQ/\cJ \cQ \cong \cQ_0 \oplus \cQ_1$, Lemma~\ref{lem:map-phi} gives a canonical map
\begin{displaymath}
\cR_0 \oplus \cR_1 \to (\cQ_0 \oplus \cQ_1) \otimes_{\cO_X/\cJ} \cJ/\cJ^2
\end{displaymath}
where $\cR_0$ maps to $\cQ_1$ and $\cR_1$ maps to $\cQ_0$. We can thus convert it into a map
\begin{displaymath}
\Phi \colon (\cQ_1^* \otimes \cR_0) \oplus (\cQ_0^* \otimes \cR_1) \to \cJ/\cJ^2.
\end{displaymath}
We claim that $\Phi$ is an isomorphism.

First, we show that $\Phi$ is injective. Note that $\Phi$ is $\bG_0$-equivariant, and that $\cQ_1^* \otimes \cR_0$ and $\cQ_0^* \otimes \cR_1$ are both irreducible homogeneous bundles which are not isomorphic to each other. So it suffices to show that $\Phi$ is not identically 0 on either summand. Using the universal property of $\Gr_{r|s}(V)$, it suffices to give a single superalgebra $T$ together with a rank $r|s$ summand $M$ of $T \otimes V$ such that the pullback of $\Phi$ to each summand is non-zero.

This can be done with $T = \bC[\epsilon] / (\epsilon^2)$ where $\epsilon$ has degree~1. Pick bases $e_1,\dots,e_n$ for $V_0$ and $f_1,\dots,f_m$ for $V_1$. Then we let $M$ be the free $T$-submodule of $T \otimes V$ with basis elements $v_i = 1 \otimes e_i + \epsilon \otimes f_1$ for $i=1,\dots,r$ and $w_j = \epsilon \otimes e_1 + 1\otimes f_j$ for $j=1,\dots,s$. Then $\phi(v_i) = \epsilon f_1$ and $\phi(w_j) = \epsilon e_1$ so that both components are indeed non-zero.

Surjectivity now follows since both sides are equivariant vector bundles of the same dimension. (Note that the rank of $\cJ/\cJ^2$ is the odd part of the dimension of $X$.)
\end{proof}

From the above, we see that we are in the setting of Theorem~\ref{thm:super-gm}. Indeed, $X$ is a smooth supervariety and $X_{\ord}$ is projective, and we have a short exact sequence
\begin{displaymath}
0 \to \cJ/\cJ^2 \to \epsilon \to \eta \to 0
\end{displaymath}
where $\epsilon=W^* \otimes \cO_{X_{\ord}}$, with $W$ as in \S \ref{s:gs}. We remind that we are assuming that $r \ge s$. From this sequence, one easily sees that $\Spec(\Sym(\eta))$ is identified with the variety $Y$. Thus the varieties $Z$ and $\tilde{Z}$ in appearing in \S \ref{ss:gm} and Theorem~\ref{thm:super-gm} match those studied in \S \ref{s:gs} in our setting. We thus find:

\begin{proposition} \label{prop:grass-ss}
There is a natural isomorphism
\begin{displaymath}
\rH^q(X, \gr^{p+q}(\cO_X)) = \Tor_p^S(\cO_{\tilde{Z}}, \bC)_{p+q}
\end{displaymath}
and spectral sequence
\begin{displaymath}
\rE^{p,q}_1 = \Tor^S_{-q}(\cO_{\tilde{Z}}, \bC)_p \implies \rH^{p+q}(X, \cO_X).
\end{displaymath}
\end{proposition}

\begin{proof}
This follows from Theorem~\ref{thm:super-gm}. The key hypothesis, that the higher cohomology of $\cO_Y$ vanishes, is provided by Theorem~\ref{thm:detvar}(d).
\end{proof}

Recall the super Euler characteristic $\chi$ defined in \S\ref{ss:super-geom}.

\begin{proposition}
  Assume that $r \ge s$. We have
  \[
    \chi(\Gr_{r|s}(\bC^{n|m})) = \begin{cases} \binom{m}{s} & \text{if $n-m \ge r-s$}\\ \binom{n}{s} & \text{if $r=s$ and $m > n$}\\ 0 & \text{otherwise} \end{cases}.
  \]
\end{proposition}

\begin{proof}
  This follows immediately from Corollary~\ref{cor:euler} and Theorem~\ref{thm:detvar}.
\end{proof}

\subsection{The Lascoux resolution} \label{ss:lascoux}

To make use of Proposition~\ref{prop:grass-ss}, we need to know something about the minimal free resolution of $\cO_{\tilde{Z}}$: specifically, we need to know its terms, i.e., the Tor groups (the differentials will not concern us). Since $\cO_{\tilde{Z}}$ is a free $\cO_Z$-module (Theorem~\ref{thm:detvar}(b)), its Tor groups are direct sums of the Tor groups of $\cO_Z$. Lascoux determined the Tor groups of $\cO_Z$, and we now review his results.

Let $a$ and $b$ be non-negative integers. Given (integer) partitions $\alpha$ and $\beta$ with $\ell(\alpha) \le b$ and $\beta_1 \le b$, define the partitions
\begin{align*}
P_{a,b}(\alpha, \beta) &= (b + \alpha_1, \dots, b + \alpha_b, b^a, \beta_1, \dots, \beta_{\ell(\beta)}),\\
Q_{a,b}(\alpha, \beta) &= (b + \beta^\dagger_1, \dots, b + \beta^\dagger_b, b^a, \alpha^\dagger_1, \dots, \alpha^\dagger_{\alpha_1}),
\end{align*}
which we visualize in terms of Young diagrams as follows:
\begin{displaymath}
  \begin{tikzpicture}[scale=.4]
    \path (-3,3.5)  node {$P_{a,b}(\alpha,\beta) = $};
    \draw (0,0) -- (0,7) -- (8,7) -- (8,6) -- (6,6) -- (6,5) -- (5,5) -- (5,4) -- (3,4) -- (3,2) -- (2,2) -- (2,0) -- (0,0) -- cycle;
\draw (0,3) -- (3,3);
\draw (0,4) -- (3,4) -- (3,7);
\path (1.5,5.5) node {$b \times b$};
\path (1.5,3.5) node {$a \times b$};
\path (4.5,5.5) node {$\alpha$};
\path (1,1.5) node {$\beta$};
\path (12,3.5) node {$Q_{a,b}(\alpha,\beta) = $};
\draw (15,-2) -- (15,7) -- (21,7) -- (21,5) -- (19,5) -- (19,4) -- (18,4) -- (18,1) -- (17,1) -- (17,0) -- (16,0) -- (16,-2) -- cycle;
\draw (15,4) -- (18,4) -- (18,7);
\draw (15,3) -- (18,3);
\path (16.5,5.5) node {$b \times b$};
\path (16.5,3.5) node {$a \times b$};
\path (19.5,6) node {$\beta^\dagger$};
\path (16.5,1.5) node {$\alpha^\dagger$};
\end{tikzpicture}
\end{displaymath}
For a partition $\lambda$, we let $\bS_{\lambda}$ denote the corresponding Schur functor. With this notation, we can state Lascoux's result:

\begin{theorem}[Lascoux] \label{thm:lascoux}
Put $a=m-\delta$. If $q=ab$ for some non-negative integer $b$ then
\begin{displaymath}
\Tor_p^S(\cO_Z, \bC)_{p+q} = \bigoplus_{\substack{\alpha, \beta\\ \ell(\alpha) \le b,\ \beta_1 \le b\\ p = b^2 + |\alpha| + |\beta|}} \bS_{P_{a,b}(\alpha, \beta)}(V_0) \otimes \bS_{Q_{a,b}(\alpha, \beta)}(V_1^*)
\end{displaymath}
as representations of $\bG_0$. If $q$ is not divisible by $b$ then $\Tor^{\cO_W}_p(\cO_Z, \bC)_{p+q}=0$.
\end{theorem}

\begin{proof}
  See \cite[Proposition 6.1.3]{weyman} which contains the description of $\Tor_p^{\cO_{W_0}}(\cO_{Z_0},\bC)$. Note that since $Z=Z_0 \times W_1$, the minimal free resolution for $\cO_Z$ over $\cO_W$ is obtained from the minimal free resolution for $\cO_{Z_0}$ over $\cO_{W_0}$ by tensoring with $\cO_{W_1}$. In particular, the Tor groups agree, and this conversion preserves the grading.
\end{proof}

\begin{corollary} \label{cor:lascoux}
The $\bG_0$-representation $\bigoplus_{p \ge 0} \Tor_p^S(\cO_Z, \bC)$ is multiplicity-free.
\end{corollary}

\subsection{Proof of Theorem~\ref{thm:grcoh}}

Put
\begin{displaymath}
\tilde{L}_k=\bigoplus_{p \ge 0} \Tor_p^S(\cO_{\tilde{Z}}, \bC)_{p+k}, \qquad
L_k = \bigoplus_{p \ge 0} \Tor_p^S(\cO_Z, \bC)_{p+k}.
\end{displaymath}
These vector spaces carry algebraic representations of $\bG_0$. In particular, they are semi-simple as $\bG_0$-representations, as are all algebraic representations.

\begin{proposition} \label{prop:typeA-syzygies}
  \addtocounter{equation}{-1}
  \begin{subequations}
    The graded vector space $\Tor^S_p(\cO_{\tilde{Z}}, \bC)$ is naturally a graded $A$-module, and the induced map
\begin{equation} \label{eq:tormap}
A \otimes_{\bC} \Tor_p^S(\cO_Z, \bC) \to \Tor_p^S(\cO_{\tilde{Z}}, \bC)
\end{equation}
is an isomorphism of graded $A$-modules.
\end{subequations}
\end{proposition}

\begin{proof}
The space $\Tor^S_p(\cO_{\tilde{Z}}, \bC)$ is naturally a module over $\Tor^S_0(\cO_{\tilde{Z}}, \bC)=\cO_{\tilde{Z}} \otimes_{\cO_Z} \bC$, which we have seen (Theorem~\ref{thm:detvar}(c)) is isomorphic to $A$. Let $A' \subset \cO_{\tilde{Z}}$ be a homogeneous $\bC$-subspace such that the map $A' \to \cO_{\tilde{Z}} \otimes_{\cO_Z} \bC=A$ is an isomorphism. Then $A'$ is a minimal generating space for $\cO_{\tilde{Z}}$ as an $\cO_Z$-module. Since $\cO_{\tilde{Z}}$ is free as an $\cO_Z$-module (Theorem~\ref{thm:detvar}(b)), it follows that the natural map
\begin{displaymath}
A' \otimes_{\bC} \cO_Z \to \cO_{\tilde{Z}}
\end{displaymath}
is an isomorphism of $\cO_Z$-modules. We thus see that the induced map
\begin{displaymath}
A' \otimes_{\bC} \Tor^S_p(\cO_Z, \bC) \to \Tor^S_p(\cO_{\tilde{Z}}, \bC)
\end{displaymath}
is an isomorphism. This map is isomorphic (in the obvious manner) to \eqref{eq:tormap}, and so \eqref{eq:tormap} is an isomorphism of vector spaces. The map \eqref{eq:tormap} is a homomorphism of graded $A$-modules simply by its definition.
\end{proof}

\begin{proposition} 
The $\bG_0$-representations $\tilde{L}_k$ and $\tilde{L}_{k+1}$ have no simple factors in common.
\end{proposition}

\begin{proof}
Proposition~\ref{prop:typeA-syzygies} shows that $\tilde{L}_k=\bigoplus_{i \ge 0} A_{k-i} \otimes_{\bC} L_i$. Since $A$ is concentrated in even degrees, we see that $\tilde{L}_k$ is a sum of $L_i$'s with $i$ of the same parity as $k$. Since $\bigoplus_{i \ge 0} L_i$ is multiplicity-free as a $\bG_0$-representation (Corollary~\ref{cor:lascoux}), the claim follows.
\end{proof} 

\begin{proposition} \label{prop:SS-degen}
  The spectral sequence in Proposition~\ref{prop:grass-ss} degenerates at the first page.
\end{proposition}

\begin{proof}
Let $E$ be the spectral sequence from Proposition~\ref{prop:grass-ss}, and put $E^k_r=\bigoplus_{p+q=k} E^{p,q}_r$, so that the differential is a map $E^k_r \to E^{k+1}_r$. We have $E^k_1=\tilde{L}_k$. Since $E^k_r$ is a subquotient of $E^k_1$, we see that $E^k_r$ is a semi-simple $\bG_0$-representation, and $E^k_r$ and $E^{k+1}_r$ have no irreducible factors in common. It follows that the differential $E^k_r \to E^{k+1}_r$ must vanish, as it is a map of $\bG_0$-representations. This completes the proof.
\end{proof}

\begin{corollary}
We have canonical $\bG_0$-equivariant isomorphisms
\begin{displaymath}
\rH^i(X, \cO_X)=\gr(\rH^i(X, \cO_X))=\tilde{L}_i
\end{displaymath}
\end{corollary}

\begin{proof}
Proposition~\ref{prop:SS-degen} gives a canonical isomorphism $\gr(\rH^i(X, \cO_X))=\tilde{L}_i$. It follows that $\gr(\rH^i(X, \cO_X))$ is multiplicity free as a representation of $\bG_0$, and so the same is true of $\rH^i(X, \cO_X)$. Thus the filtration on $\rH^i(X, \cO_X)$ canonically splits, which yields a canonical isomorphism $\rH^i(X, \cO_X)=\gr(\rH^i(X, \cO_X))$.
\end{proof}

\begin{proposition} \label{prop:G=G0-invt}
We have $\rH^i(X, \cO_X)^{\bG}=\rH^i(X, \cO_X)^{\bG_0}$.
\end{proposition}

\begin{proof}
First assume that $m>\delta$. Let $E=\rH^i(X, \cO_X)^{\bG_0}$. The action of the upper-triangular nilpotent piece of the Lie algebra of $\bG$ on $W$ is a $\bG_0$-equivariant map
\begin{displaymath}
V_1^* \otimes V_0 \otimes E \to \rH^i(X, \cO_X).
\end{displaymath}
First note that $\rH^i(X, \cO_X)$ does not contain a $\bG_0$-subrepresentation isomorphic to $V_1^* \otimes V_0$: in the notation of Theorem~\ref{thm:lascoux}, this is only possible if $P_{a,b}(\alpha,\beta) = Q_{a,b}(\alpha,\beta) = (1)$, which is only possible if $a=m-\delta=0$. It follows that this map must be~0. Similarly, the corresponding map for the lower-triangular piece is~0. Thus $E$ is annihilated by the Lie algebra of $\bG$. It follows that $\bG$ acts trivially on $E$.

Finally, if $m=\delta$, this means that $r=n$ and $s=0$. In this case, $X$ is a point and its structure sheaf is the exterior algebra on $V_1^* \otimes V_0$. The $\bG_0$-invariant space is spanned by the unit element, and since $\bG$ acts via algebra automorphisms, the unit must also be invariant under $\bG$.
\end{proof}

Let $U$ be a graded vector space. We define the \defn{trivial filtration} on $U$ by $\Fil^i(U)=\bigoplus_{j \ge i} U_j$. With respect to this filtration, we have a natural isomorphism $U=\gr(U)$.

\begin{proposition}
We have a natural isomorphism $\rH^*(X, \cO_X)^{\bG}=A$ of graded algebras. Moreover, the filtration on $\rH^*(X, \cO_X)$ induces the trivial filtration on $A$.
\end{proposition}

\begin{proof}
Let $B=\rH^*(X, \cO_X)^{\bG}$. By Proposition~\ref{prop:G=G0-invt}, we have $B=\rH^*(X, \cO_X)^{\bG_0}$. We have
\begin{align*}
\gr^p(B_q)
&=\gr^p(\rH^q(X, \cO_X))^{\bG_0}=\rH^q(X, \gr^p(\cO_X))^{\bG_0} \\
&= \Tor_{p-q}^S(\cO_{\tilde{Z}}, \bC)_p^{\bG_0} = \bigoplus_{i+j=p} A_i \otimes \Tor_{p-q}^S(\cO_Z, \bC)_j^{\bG_0}.
\end{align*}
In the first step we used the the $\bG_0$ action is semi-simple; in the second, that the spectral sequence degenerates (Proposition~\ref{prop:SS-degen}); in the third, Proposition~\ref{prop:grass-ss}; and in the fourth, Proposition~\ref{prop:typeA-syzygies}. Now, it is easy to see directly that $\Tor_k^S(\cO_Z, \bC)^{\bG_0}$ vanishes for $k \ne 0$, and that $\Tor_0^S(\cO_Z, \bC)^{\bG_0}$ is one-dimensional and concentrated in degree~0; this can also be read off of Theorem~\ref{thm:lascoux}. We thus find that
\begin{displaymath}
\gr^p(B_q) = \begin{cases}
A_p & \text{if $p=q$} \\
0 & \text{otherwise} \end{cases}.
\end{displaymath}
This shows that $\gr(B)=A$, and the isomorphism is one of rings by Remark~\ref{rmk:Tor-alg}. It also shows that the filtration on $B$ is trivial, and so $B$ is isomorphic to $\gr(B)$ as a ring.
\end{proof}

We regard $\rH^*(X, \cO_X)$ as a graded $A$-module via the above proposition. We have isomorphisms of $\bG_0$-representations
\begin{displaymath}
\rH^i(X, \cO_X) \cong \gr(\rH^i(X, \cO_X)) = \tilde{L}_i = \bigoplus_{j \ge 0} A_{i-j} \otimes L_j.
\end{displaymath}
It follows that $\rH^i(X, \cO_X)$ contains a unique $\bG_0$-subrepresentation isomorphic to $L_i$. (Corollary~\ref{cor:lascoux} is important here.) Call this subspace $E^i$, and let $E=\bigoplus_{i \ge 0} E^i$.

\begin{proposition}
The natural map $A \otimes E \to \rH^*(X, \cO_X)$ is an isomorphism.
\end{proposition}

\begin{proof}
We have
\begin{displaymath}
\gr(\rH^*(X, \cO_X)) = \bigoplus_{p,q \ge 0} \Tor_p^S(\cO_{\tilde{Z}}, \bC)_q = A \otimes \bigoplus_{p,q \ge 0} \Tor_p(\cO_Z, \bC)_q = A \otimes \bigoplus_{i \ge 0} L_i.
\end{displaymath}
We note that the $L_i$ in the final direct sum belongs to $\gr(\rH^i(X, \cO_X))$. We thus see that $\gr(\rH^*(X, \cO_X))$ is a free $A$-module (this relies on Remark~\ref{rmk:Tor-alg}). Moreover, if $\ol{E}{}^i \subset \gr(\rH^i(X, \cO_X))$ is the associated graded of $E^i$ then $\ol{E}{}^i$ corresponds to the copy of $L_i$ in the final direct sum above. Letting $\ol{E}=\bigoplus_{i \ge 0} \ol{E}{}^i$, we thus see that the natural map $A \otimes \ol{E} \to \gr(\rH^*(X, \cO_X))$ is an isomorphism. The result follows from this.
\end{proof}

\begin{proposition} \label{prop:Gext}
Suppose that $M$ and $M'$ are representations of $\bG$ such that $M \cong L_i$ and $M' \cong L_j$ as $\bG_0$-representations, with $i \ne j$. Then $\Ext^1_{\bG}(M, M')=0$.
\end{proposition}

\begin{proof}
Consider an extension
\begin{displaymath}
0 \to M' \to N \to M \to 0.
\end{displaymath}
Since $i \ne j$, there is a unique $\bG_0$-splitting of this sequence. We thus regard $M$ as a $\bG_0$-subrepresentation of $N$. Consider the map
\begin{displaymath}
V_1^* \otimes V_0 \otimes M \to N
\end{displaymath}
giving the action of the upper-triangular nilpotent piece of the Lie algebra of $\bG$ on $M$. Since this map is $\bG_0$-equivariant, it follows from Theorem~\ref{thm:lascoux} and Pieri's rule \cite[Equation (6.8)]{fultonharris} that this map is zero. Similarly for the lower-triangular piece. Thus $M$ is $\bG$-subrepresentation of $N$, which completes the proof.
\end{proof}

\begin{proposition} \label{prop:E-subrep}
$E$ is a $\bG$-subrepresentation of $\rH^*(X, \cO_X)$.
\end{proposition}

\begin{proof}
Put $H^i=\rH^i(X, \cO_X)$. We prove that $E^i$ is a $\bG$-subrepresentation of $H^i$ by induction on $i$. Thus suppose that $i$ is given and $E^j$ is a $\bG$-subrepresentation of $H^j$ for all $j<i$. Let $X=\sum_{j=0}^{i-1} A_{i-j} E^j$, which is a $\bG$-subrepresentation of $H^i$ by the inductive hypothesis, and let $Y=H^i/X$. We thus have a short exact sequence
\begin{displaymath}
0 \to X \to H^i \to Y \to 0
\end{displaymath}
of $\bG$-representations. Since the map $A \otimes E \to H^*$ is an isomorphism, we see that $X \cong \bigoplus_{j=0}^{i-1} A_{i-j} \otimes E^j$ and that $E^i$ is a complementary subspace to $X$ in $H^i$. In particular, the map $E^i \to Y$ is an isomorphism of $\bG_0$-representations. It now follows from Proposition~\ref{prop:Gext} that $\Ext^1_{\bG}(Y, X)=0$. Indeed, we have $E^j \cong L_j$ and $Y \cong L_i$ as $\bG_0$-representations and $i \ne j$. We thus see that there is a $\bG$-splitting $\sigma \colon Y \to H^i$. Since $E^i$ is the unique $\bG_0$-subrepresentation of $H^i$ isomorphic to $L_i$, we see that $\sigma$ is unique and $\sigma(Y)=E^i$. Thus $E^i$ is a $\bG$-subrepresentation, as required.
\end{proof}

\subsection{The $\bG$ action on the Lascoux resolution} \label{ss:pw}

Our original motivation for this work was to give a geometric explanation for the result of \cite{pragacz-weyman} (and \cite{sam}) that the linear strands of the Tor groups of the determinantal variety $Z_0$ carry an action of $\bG$. Of course, it is equivalent to do this for $Z$, since $Z$ and $Z_0$ have the same Tor groups. This follows from Proposition~\ref{prop:E-subrep}, as $E^i=L_i$ is a $\bG$-subrepresentation of $\rH^i(X, \cO_X)$. If one is simply interested in constructing the $\bG$ action on $L_i$, it suffices to consider the case where $s=0$; here we have
\begin{displaymath}
L_i=\rH^i(X, \cO_X),
\end{displaymath}
and so the obvious $\bG$ action on the right induces one on the left.

\begin{remark} 
The action of $\bG$ of $L_i$ induces an action of its Lie superalgebra on $L_i$. With respect to the decomposition $\fgl(n|m) = \Hom(\bC^n, \bC^m) \oplus (\fgl(n) \times \fgl(m)) \oplus \Hom(\bC^m, \bC^n)$, the middle piece corresponds to the obvious action of $\GL(n) \times \GL(m)$ on determinantal varieties, while one of the other pieces (depending on conventions) corresponds to the action of the Tor algebra on linear strands by Remark~\ref{rmk:Tor-alg}. With that structure in place, the $L_i$ are highest weight representations of $\bG$, and if $m>\delta$, the results of \cite{pragacz-weyman, sam} show that they are irreducible (and hence there is a unique way to extend the action of 2 out of the 3 pieces to the third). If $m=\delta$, then $L_i$ is only non-zero for $i=0$, in which case it is a Koszul complex, which we can understand as a Kac module, i.e.,  induced from the trivial representation of $\Hom(\bC^n, \bC^m) \oplus (\fgl(n) \times \fgl(m))$ (again depending on conventions). It is indecomposable, but not irreducible in general.
\end{remark}

\end{document}